\documentclass[a4paper,american,reqno]{amsart}

\usepackage{babel}
\usepackage[T1]{fontenc}
\usepackage[utf8]{inputenc}
\usepackage{siunitx}
\usepackage{booktabs}
\usepackage{multirow}
\usepackage{xspace}
\usepackage{relsize}
\usepackage{accents}
\usepackage{csvsimple}
\usepackage{csquotes}
\usepackage{microtype}
\usepackage{amsfonts}
\usepackage{amssymb}
\usepackage{pdflscape}
\usepackage{csquotes}
\usepackage{bbold}
\usepackage{enumitem}
\usepackage{placeins}
\usepackage[ruled,linesnumbered]{algorithm2e}
\usepackage[style=numeric,
            giveninits=true,
            maxbibnames=100,
            maxcitenames=2,
            isbn=false,
            doi=true,
            url=false,
            isbn=false,
            sortcites=true,
            backend=biber]{biblatex}
\usepackage[colorlinks,
            citecolor=blue,
            urlcolor=blue,
            linkcolor=blue]{hyperref} 

\makeatletter
\patchcmd{\@settitle}{\uppercasenonmath\@title}{\scshape\large}{}{}
\patchcmd{\@setauthors}{\MakeUppercase}{\scshape\normalsize}{}{}
\makeatother

\vfuzz \hfuzz
\makeatletter
\@namedef{subjclassname@2020}{\textup{2020} Mathematics Subject Classification}
\makeatother

\newcommand{\abbr}[1][abbrev]{#1.\xspace}

\newcommand{\eg}{\abbr[e.g]}

\newcommand{\ie}{\abbr[i.e]}

\newcommand{\Wlog}{\abbr[w.l.o.g]}

\newcommand{\field}{\mathbb}

\newcommand{\reals}{\field{R}}

\newcommand{\R}{\reals}

\newtheorem{lemma}{Lemma}
\newtheorem{remark}{Remark}
\newtheorem{theorem}{Theorem}

\newtheorem{example}{Example}
\newtheorem{corollary}{Corollary}
\newtheorem{proposition}{Proposition}

\newtheorem{observation}{Observation}

\newcommand{\nodes}{V}

\newcommand{\leaves}{L}
\newcommand{\node}{v}
\newcommand{\othernode}{u}
\newcommand{\parent}{p}
\newcommand{\sibling}{s}

\newcommand{\predecessors}{P}

\newcommand{\st}{\text{s.t.}}

\makeatletter
\newcommand{\fcdot}{\,\cdot\,}
\newcommand{\fcarg}[1]{\def\fc@rg{#1}\ifx\fc@rg\empty\fcdot\else\fc@rg\fi}
\makeatother
\newcommand{\abs}[1]{\lvert\fcarg{#1}\rvert}

\newcommand{\card}{\abs}

\newcommand{\norm}[2][]{\lVert\fcarg{#2}\rVert\ifx#1\empty\else_{#1}\fi}
\newcommand{\Norm}[2][]{\left\lVert#2\right\rVert\ifx#1\empty\else_{#1}\fi}

\newcommand{\snorm}[2][]{\lvert\!\lvert\!\lvert
  \fcarg{#1}\rvert\!\rvert\!\rvert\ifx#2\empty\else_{#1}\fi}
\newcommand{\Snorm}[2][]{\left\lvert\!\left\lvert\!\left\lvert
  #2\right\rvert\!\right\rvert\!\right\rvert\ifx#1\empty\else_{#1}\fi}
\newcommand{\sprod}[3][]{\langle\fcarg{#2},\fcarg{#3}\rangle\ifx#1\empty\else_{#1}\fi}
\newcommand{\Sprod}[3][]{\left\langle\fcarg{#2},\fcarg{#3}\right\rangle\ifx#1\empty\else_{#1}\fi}

\newcommand{\optmathindex}[1]{\ifx#1\empty\else,#1\fi}
\newcommand{\opttextindex}[1]{\ifx#1\empty\else,\text{#1}\fi}
\newcommand{\optmathsb}[1]{\ifx#1\empty\else_{#1}\fi}
\newcommand{\opttextsb}[1]{\ifx#1\empty\else_{\text{#1}}\fi}
\newcommand{\optmathsp}[1]{\ifx#1\empty\else^{#1}\fi}
\newcommand{\opttextsp}[1]{\ifx#1\empty\else^{\text{#1}}\fi}

\newcommand{\continuousFunctions}[1]{\mathcal{C}\ifx#1\empty\else^{#1}\fi}

\newcommand{\piecewiseContinuousFunctions}[1]{\mathcal{C}_p\ifx#1\empty\else^{#1}\fi}

\newcommand{\define}{\mathrel{{\mathop:}{=}}}

\newcommand{\objvect}{c}
\newcommand{\decvar}{x}
\newcommand{\appdecvar}{\bar{\decvar}}
\newcommand{\decvarsize}{n}
\newcommand{\binvarsize}{\decvarsize_b}
\newcommand{\consmat}{A}
\newcommand{\conslhs}{b}

\newcommand{\varset}{I}
\newcommand{\binindex}{i}

\newcommand{\feasset}{F}
\newcommand{\lpfeasset}{\hat{F}}
\newcommand{\tree}{G}
\newcommand{\edges}{E}

\newcommand{\lpval}{l}
\newcommand{\onebranchset}{\mathbb{1}}
\newcommand{\zerobranchset}{\mathbb{0}}

\newcommand{\decvarset}{\mathcal{X}}
\newcommand{\permu}{\phi}
\newcommand{\appset}{O}

\newcommand{\bal}{1}
\newcommand{\disjouterapprox}{\appset_\bal}

\newcommand{\msset}{H}
\newcommand{\indvar}{z}

\newcommand{\minlp}{\mu}
\newcommand{\binaryset}{I_{\mathrm{B}}}
\newcommand{\contset}{I_{\mathrm{C}}}

\newcommand{\proj}{\mathrm{proj}}
\newcommand{\BB}{BB\xspace}
\newcommand{\conv}{\mathrm{conv}}

\newcommand{\child}{c}
\newcommand{\rootn}{r}
\newcommand{\nmsset}{M}

\newcommand{\treeind}{2}
\newcommand{\msind}{3}
\newcommand{\onelbset}{\nodes_{1}}
\newcommand{\zerolbset}{\nodes_{0}}

\newcommand{\PTwoMsetLP}{\hat{\msset}}

\newcommand{\decvarsetLP}{\hat{\mathcal{X}}}
\newcommand{\STI}{\text{STI}}
\newcommand{\viparam}{\Delta}
\newcommand{\CGLP}{CGLP\xspace}

\newcommand{\atom}{\mathcal{A}}
\newcommand{\atommat}{D}
\newcommand{\atomrhs}{f}

\newcommand{\atomconssize}{k}

\newcommand{\cglpvar}{u}
\newcommand{\cglpmult}{\pi}

\newcommand{\order}{\mathcal{O}}

\newcommand{\nmssetLP}{\hat{\nmsset}}
\newcommand{\clos}{\mathcal{C}}
\newcommand{\MBP}{MBP\xspace}
\newcommand{\MBPs}{MBPs\xspace}
\newcommand{\treeinf}{\mathcal{G}}

\newcommand{\mixedbinset}{Q}

\newcommand{\pathflow}{\psi}
\newcommand{\capa}{\eta}
\newcommand{\posnodes}{\nodes^{>}}
\newcommand{\genval}{\nu}
\newcommand{\dualsettree}{\Omega_{\treeind}}
\newcommand{\dualsetms}{\Omega_{\msind}}

\newcommand{\sepvaltree}{\genval_{\treeind}}
\newcommand{\sepvalms}{\genval_{\msind}}
\newcommand{\dualroot}{\alpha}
\newcommand{\dualflow}{\beta}
\newcommand{\dualcapa}{\lambda}
\newcommand{\dualmix}{\rho}

\newcommand{\flowset}{N}
\newcommand{\flowsetLP}{\hat{N}}
\newcommand{\OtwoApp}{\appset_{\treeind}}
\newcommand{\OthreeApp}{\appset_{\msind}}
\newcommand{\HNhat}{\PTwoMsetLP \cap \flowsetLP}
\newcommand{\NMhat}{\flowsetLP \cap \nmssetLP}
\newcommand{\mixcapa}{\gamma}

\newcommand{\lowerb}{\underline{b}}
\newcommand{\upperb}{\overline{b}}

\usepackage{pgfplots}
\usepackage{pgfplotstable}
\usepgfplotslibrary{groupplots} \usepgfplotslibrary{fillbetween}
\pgfplotsset{compat=1.17}
\usetikzlibrary{positioning}
\usetikzlibrary{arrows.meta}
\definecolor{cO1}{HTML}{2A78D6}
\definecolor{cO2}{HTML}{EB6834}
\definecolor{cO3}{HTML}{1BAF7A}
\definecolor{cC}{HTML}{EDA100}
\definecolor{cQ}{HTML}{222222}

 \bibliography{references}

\begin{document}

\title[Polyhedral Outer Approximations from Branch-and-Bound Trees]{A Hierarchy of Polyhedral Outer Approximations from Branch-and-Bound Trees}

\author[M. Roland, N. Sugishita, A. Forel, Y. Emine, R. Fukasawa, T. Vidal]{Marius Roland, Nagisa Sugishita, Alexandre Forel,\\
   Youssouf Emine, Ricardo Fukasawa, Thibaut Vidal}

\address[Marius Roland]{(A) Univ. Lille,
   Inria,
   CNRS,
   Centrale Lille,
   UMR 9189 CRIStAL,
   Lille, F-59000,
   France;
   (B) Polytechnique Montréal,
   André-Aisenstadt Pavillon,
   2920 Tour Road, Montreal,
   Quebec H3T 1N8, Canada
 }
\email{mmmroland@gmail.com}

\address[Nagisa Sugishita]{Université de Montréal, André-Aisenstadt Pavillon, 2920 Tour Road, Montreal, Quebec H3T 1N8, Canada}
\email{nagisa.sugishita@umontreal.ca}

\address[Alexandre Forel]{Polytechnique Montréal, André-Aisenstadt Pavillon, 2920 Tour Road, Montreal, Quebec H3T 1N8, Canada\\\emph{The work was conducted prior to joining Amazon.}}
\email{alex.forel@gmail.com}

\address[Youssouf Emine]{Polytechnique Montréal, André-Aisenstadt Pavillon, 2920 Tour Road, Montreal, Quebec H3T 1N8, Canada}
\email{youssouf.emine@gmail.com}

\address[Ricardo Fukasawa]{University of Waterloo, 200 University Avenue West, Waterloo, Ontario N2L3G1, Canada}
\email{rfukasawa@uwaterloo.ca}

\address[Thibaut Vidal]{Polytechnique Montréal, André-Aisenstadt Pavillon, 2920 Tour Road, Montreal, Quebec H3T 1N8, Canada}
\email{thibaut.vidal@polymtl.ca}

\date{\today}

\begin{abstract}

  This paper investigates the a posteriori analysis of Branch-and-Bound~(\BB) trees to extract structural information about the feasible region of mixed-binary linear programs. Using only the branching decisions of a tree and one dual bound per node, we introduce three polyhedral outer approximations of the feasible region, namely a disjunctive, a network-flow-based, and a mixing-set-based approximation. We show that they are the projections of the convex hull and of two continuous relaxations of a single mixed-binary set, and that they form an inclusion hierarchy. Separating over the network-flow-based and the mixing-set-based approximations reduces to a minimum-cost flow problem and to a maximum flow problem on a network obtained from the tree. Both are solved in strongly polynomial time, and their optimal dual solutions provide separating inequalities. We further introduce star tree inequalities, which generalize the strong-branching inequalities of K{\i}l{\i}n\c{c} et al.\ to \BB trees of arbitrary depth and are separated in strongly polynomial time. Finally, we show that every inclusion of the hierarchy can be strict.

\end{abstract}

\keywords{Branch-and-bound, mixed-binary programming, disjunctive programming, network flows, mixing sets, valid inequalities}
\subjclass[2020]{90C11, 90C57, 90C35}

\maketitle

\section{Introduction}
\label{sec:introduction}

The Branch-and-Bound~(\BB) method proposed by~\textcite{c317b127-1fe2-3862-8218-ad5f3daf32e7}, has transformed mixed-integer linear optimization. This central role has motivated extensive research on the design and analysis of branch-and-bound and its many components, notably branching rules, valid inequalities, and primal heuristics. Over time, additional features such as presolve, symmetry handling, parallelization, and restarts have become standard. Each of these elements continues to be the focus of significant attention~\cite{clautiaux2025last}.

Despite their practical effectiveness, \BB algorithms are difficult to configure. Each component admits many alternatives, the components interact in subtle ways, and the best configuration varies from one instance to another. As a result, the default settings of state-of-the-art solvers rely heavily on expert judgment and extensive empirical tuning~\cite{bolusani2024scip}. This reliance on heuristic selection sits uneasily with the prescriptive spirit of optimization and has motivated a recent body of work that seeks a rigorous, mathematical understanding of branch-and-bound. Recent research seeks to address fundamental questions, such as identifying classes of instances that are solvable efficiently using BB~\cite{dey2023lower}, and analyzing the circumstances under which specific algorithmic components are most or least effective~\cite{dey2024theoretical}.

A \BB tree records more about an instance than an optimal solution and the global dual bound that certifies its optimality. Its branching decisions partition the feasible region of the underlying optimization problem, and each of its nodes carries a dual bound that is valid on the corresponding subset of the feasible region. Such trees are produced in many settings: every solve that certifies optimality produces a complete tree, interrupted solves and restarts leave partial trees~\cite{contardo2023cutting}, strong branching builds a small tree at every node of the search~\cite{kilincc2014strong}, and some methods build trees on purpose to generate cutting planes~\cite{balas2026vpolyhedral, chen2011finite}. In most cases, however, a tree is discarded once the instance is solved, together with the information it carries.

This paper studies that information a posteriori. We consider a mixed-binary linear program and a \BB tree built while solving it. This tree may only be built partially, in which case it does not certify optimality. The information we use consists of the branching decisions of the tree and of the dual bounds of its nodes. We ask how it can be turned into polyhedral outer approximations of the feasible region and into valid inequalities that are compact and separated by combinatorial algorithms. We compare these objects in terms of their relative strength, measured by inclusion, and of the size and nature of their separation problems.

\subsection{Contributions.}

\begin{enumerate}
    \item We introduce three outer approximations of the feasible region of mixed-binary programs that can be systematically constructed from a \BB tree, using only the branching decisions and one dual bound per node. The first approximation is derived from disjunctive programming~\cite{balas2018disjunctive}, the second from a network flow formulation of the \BB tree, and the third from the strengthened mixing set of~\textcite{luedtke2010integer} (see also~\cite{gunluk2001mixing,atamturk2000mixed}). We further show that the three approximations are the projections of the convex hull and of two continuous relaxations of a single mixed-binary set constructed from the tree.

    \item We establish an inclusion hierarchy among the approximations, which is summarized in Figure~\ref{fig:hierarchy} and stated in Theorem~\ref{thm:hierarchy}. An example shows that all inclusions of this hierarchy can be strict, and that the mixing-set-based approximation and the closure of the inequalities introduced below are not comparable.

    \item We develop separation procedures for each outer approximation in the space of the original variables. For the disjunctive approximation, we adapt the cut-generating linear program framework~\cite{balas1996mixed,balas1993lift}. For the network-flow-based and the mixing-set-based approximations, the separation problem is a minimum-cost flow problem and a maximum flow problem on a network obtained from the \BB tree, respectively. Both problems are solved by combinatorial algorithms in strongly polynomial time. The feasible regions of the duals of these two flow problems do not depend on the point to be separated, so that both approximations are described by fixed families of inequalities in the space of the original variables.

    \item We introduce a new class of valid inequalities, which we call \emph{star tree inequalities}, and we show that they can be separated by a combinatorial algorithm in strongly polynomial time. These inequalities generalize the strong-branching inequalities of~\textcite{kilincc2014strong} from trees of depth one to trees of arbitrary depth. Their closure contains the network-flow-based approximation.
\end{enumerate}

\begin{figure}[t]
  \centering
  \renewcommand{\arraystretch}{1.1}\begin{tabular}{@{}c@{\hspace{3mm}}c@{\hspace{3mm}}c@{\hspace{3mm}}c@{\hspace{3mm}}c@{\hspace{3mm}}c@{\hspace{3mm}}c@{\hspace{3mm}}c@{\hspace{3mm}}c@{}}
    $\feasset$ & $\subseteq$ & $\disjouterapprox$ & $\subseteq$ & $\OtwoApp$ & $\subseteq$ & $\OthreeApp$ & $\cap$ & $\clos (\treeinf)$ \\[1mm]
    \footnotesize feasible & & \footnotesize disjunctive & & \footnotesize network flow & & \footnotesize mixing set & & \footnotesize star tree \\
    \footnotesize region & & \footnotesize (Section~\ref{sec:disj-outer-appr}) & & \footnotesize (Section~\ref{sec:mixing-set-outer-approximation}) & & \footnotesize (Section~\ref{sec:valid-inequalities}) & & \footnotesize inequalities \\
    & & & & & & & & \footnotesize (Section~\ref{sec:star-ineq-insp})
  \end{tabular}
  \caption{The inclusion hierarchy of Theorem~\ref{thm:hierarchy} between the feasible region~$\feasset$ and its four outer approximations obtained from a \BB tree. The sets~$\OthreeApp$ and~$\clos (\treeinf)$ are not comparable, and every inclusion can be strict.}
  \label{fig:hierarchy}
\end{figure}

\subsection{Relevant literature.}
Our work takes place within a recent and growing body of research that seeks to understand the mathematical and computational properties of \BB. Foundational results establish that even seemingly simple integer programs require \BB trees of exponential size~\cite{jeroslow1974trivial,dey2023lower}, and that selecting a branching disjunction or approximating the size of a smallest tree are hard problems~\cite{mahajan2010complexity,glaser2024computing}. Complementary results characterize the instances on which \BB performs well, the merits of individual algorithmic components such as strong branching and general disjunctions, and the relative strength of \BB and cutting planes~\cite{chvatal1980hard,dey2021branch,dey2024theoretical,owen2001disjunctive,basu2023complexity,cornuejols2025branch,fleming2021power}.

While the aforementioned literature provides a deep understanding of \BB trees, our work differs by focusing on the a posteriori extraction of structural information from a single, already computed tree. Several lines of research reuse the information contained in a \BB tree. Trees are archived as certificates that make the result of a solve independently verifiable~\cite{cheung2017verifying}, they are collected to train branching rules~\cite{gasse2019exact}, and their size is estimated during the search~\cite{hendel2022estimating} or reduced afterwards~\cite{munoz2025compressing}. We position our contribution with respect to them next.

The work of~\textcite{munoz2025compressing} on tree compression also performs an a posteriori analysis of a \BB tree. Their goal is to compress the tree into a smaller one with an equivalent or stronger dual bound, which can serve as a more compact certificate of optimality or help identify strong disjunctions. The work of~\textcite{kilincc2014strong} is perhaps the most related to ours in spirit. The authors propose using the \enquote{discarded} information from strong-branching to generate valid inequalities. This is similar to our goal of using dual bounds to create valid inequalities. However, their method extracts information during the node selection process of an active \BB search, whereas our approach performs an a posteriori analysis on an entire, static \BB tree that is already generated. Moreover, their inequalities are derived from a single branching, that is, from a \BB tree of depth one. The star tree inequalities that we introduce generalize them from a single branching to a \BB tree of arbitrary depth.

\BB trees are also used as a source of disjunctions for cut generation. \textcite{balas2026vpolyhedral} generate disjunctive cuts from the leaf nodes of a partial \BB tree. For each leaf node, they describe a relaxation by points and rays, typically obtained from the optimal basis of the node. They then generate cuts by solving a linear program in this point-ray representation, which is substantially smaller than a cut-generating linear program in the higher-dimensional space. \textcite{chen2011finite,chen2012computational} generate cuts from the multi-term disjunctions of a tree that their algorithm constructs, in which every term carries the constraints of the original problem. Our disjunctive approximation relies on the same disjunction over the leaf nodes, but each of its terms only contains the branching decisions of the leaf node, the variable bounds, and the dual bound of the node. The terms of~\textcite{balas2026vpolyhedral} are therefore at least as tight as ours, since the reduced costs of the optimal basis of a leaf node certify its dual bound. In exchange, the network-flow-based and mixing-set-based formulations have a size that does not depend on the number of variables of the original problem, exhibit a network flow structure, and admit explicit descriptions of their projections in the space of the original variables.

The network flow formulation that underlies two of our approximations also relates our work to decision diagrams. Network flow formulations over decision diagrams have been used to generate cutting planes~\cite{becker2005bdds,behle2007binary,tjandraatmadja2019target} and outer approximations~\cite{davarnia2021outer}. A \BB tree is a particular decision tree, in which every internal node branches on one binary variable. It differs from the decision diagrams used in these works in three respects. First, the order in which the variables are fixed may vary from one branch to another, and a leaf node only fixes the variables that were branched on. Second, no two nodes are merged. Third, the objective function is not encoded through arc costs but through the dual bounds of the leaf nodes, and the continuous variables only appear in the inequality that relates the objective value to these bounds. Since a variable may be fixed at several nodes of the tree, our formulations link it to the flow through the capacity of each of these arcs separately.

Finally, postoptimality analysis, which derives information about modified instances once an instance is solved, is a classical topic. \textcite{schrage1985sensitivity} describe which information should be recorded during the solve to bound the optimal value of instances with a modified right-hand side or objective, \textcite{dawande2000inference} characterize the perturbations for which the proof of optimality encoded in the tree remains valid, and \textcite{hadzic2006postoptimality} and \textcite{serra2020compact} represent the near-optimal solutions of an instance with decision diagrams to support postoptimality analysis. These works study modified instances, whereas we derive inequalities that are valid for the feasible region itself.

\subsection{Outline.}

In Section~\ref{sec:disj-outer-appr}, we present the disjunctive outer approximation. In Section~\ref{sec:mixing-set-outer-approximation}, we present the network-flow-based outer approximation. In Section~\ref{sec:valid-inequalities}, we present the mixing-set-based outer approximation. In Section~\ref{sec:star-ineq-insp}, we present the star tree inequalities. In Section~\ref{sec:overview}, we describe the approximations through functions of the binary variables and present an example in which all inclusions of the hierarchy are strict. Section~\ref{sec:conclusion} concludes. The appendices contain the networks of Corollaries~\ref{cor:sep-tree} and~\ref{cor:sep-ms}, the proof of Lemma~\ref{lem:validity-lb-indvar}, and the computation of the values of the example of Section~\ref{sec:overview}.

\section{A Disjunctive Programming based Outer Approximation}
\label{sec:disj-outer-appr}

In this section, we introduce an outer approximation based on disjunctive programming. We begin by defining the approximation and formally proving its validity. We then describe the use of a Cut-Generating Linear Program (CGLP) to derive valid inequalities in the original variable space. Finally, we examine the dimensions of this CGLP to motivate the approximations of the next sections, whose separation problems are solved by combinatorial algorithms.
\subsection{Notation and Assumptions}
\label{sec:notation}
We first introduce the notation for the class of optimization problems under study. We consider Mixed-Binary linear Problems~(\MBPs) in the following standard form:
\begin{equation}
    \label{eq:generic-mip}
    \begin{alignedat}{2}
        \min_{\decvar \in \R^\decvarsize} & \quad \objvect^\top \decvar \\
        \st & \quad \consmat \decvar \geq \conslhs, \\
            & \quad \decvar_\binindex \in \{0, 1\}, && \quad \binindex \in \binaryset, \\
            & \quad \lowerb_\binindex \leq \decvar_\binindex \leq \upperb_\binindex, && \quad \binindex \in \contset.
    \end{alignedat}
\end{equation}
Here,~$\objvect \in \reals^\decvarsize$ is the objective vector,~$\consmat \in \reals^{m \times \decvarsize}$ is the constraint matrix, and~$\conslhs \in \reals^m$ is the right-hand side vector. The vector~$\decvar \in \reals^\decvarsize$ represents the decision variables. The index set of these variables~$\varset \define \{1, \dots, \decvarsize\}$ is partitioned into a set for binary variables,~$\binaryset = \{1, \dots, \binvarsize\}$, and a set for continuous variables,~$\contset = \{\binvarsize + 1, \dots, \decvarsize\}$. The continuous variables are bounded by~$\lowerb_j$ and~$\upperb_j$, which may be infinite. The feasible region of Problem~\eqref{eq:generic-mip} is given by~$\feasset \define \{ \decvar \in \decvarset: \consmat \decvar \geq \conslhs \} $, where~$\decvarset$ is the set constructed by considering only the variable bounds, \ie,
\begin{equation*}
  \decvarset \define
  \left\{
    \decvar \in \reals^\decvarsize : \
    \begin{aligned}
      &\decvar_\binindex \in \{0,1\}, \quad \binindex \in \binaryset, \\
      &\lowerb_\binindex \leq \decvar_\binindex \leq \upperb_\binindex, \quad
        \binindex \in \contset.
    \end{aligned}
  \right\}.
\end{equation*}
Throughout this document, we use the \enquote{hat} notation, as in~$\lpfeasset$, to denote the continuous relaxation of a set. For instance,~$\lpfeasset$ is the feasible region of the continuous relaxation of Problem~\eqref{eq:generic-mip}, where the binary constraints~$\decvar_\binindex \in \{0, 1\}$ are relaxed to~$\decvar_\binindex \in [0,1]$ for all $\binindex \in \binaryset$.

Next, we introduce the notation and assumptions regarding the \BB tree considered.
The \BB method solves an \MBP by constructing a search tree,~$\tree = (\nodes, \edges)$, where~$\nodes$ is the set of nodes and~$\edges$ is the set of edges. Each node~$\node \in \nodes$ corresponds to a subproblem of the original \MBP, while each edge represents a branching decision that partitions the feasible set of the parent node's subproblem.

We make two common assumptions regarding the \BB tree~$\tree$ that we consider:
\begin{enumerate}
\item Branching is restricted to elementary binary branching: at every non-leaf node, a binary variable is fixed to zero in one child and to one in the other.
\item Every node has either zero or two children. Nodes with zero children are the leaf nodes, which we denote by the set $\leaves \subseteq \nodes$.
\end{enumerate}
These assumptions imply that the leaf nodes $\leaves$ induce a partition of the binary variable space and, by extension, of the feasible region $\feasset$. Following standard literature terminology~\cite{bondy1976graph}, we refer to a tree with this structure as a \enquote{full binary tree}.

We remark that considering full binary trees is not restrictive in practice. All results of this paper rely on a single consequence of the two assumptions above: every point~$\decvar \in \feasset$ satisfies the branching decisions of exactly one leaf node. A recorded \BB tree may contain nodes with a single child, for instance when a pruned child is not recorded. Such a tree is completed by adding, for each of these nodes, the missing child as a leaf node.

We also define notation capturing the relationships between the nodes of the tree, which is summarized in Table~\ref{tab:tree-notation}. For any non-leaf node~$\node \in \nodes \setminus \leaves$, we denote by~$\child_0(\node)$ and~$\child_1(\node)$ its two children, in which the variable branched on at~$\node$ is fixed to zero and to one, respectively. For any non-root node~$\node \in \nodes \setminus \{\rootn\}$, we denote by~$\parent(\node)$ and~$\sibling(\node)$ its parent and its sibling, and by~$\binindex(\node) \in \binaryset$ the index of the variable branched on at~$\parent(\node)$. The sets
\begin{equation}
  \label{eq:def-lbset}
  \zerolbset \define \{\child_0 (\othernode) : \ \othernode \in \nodes \setminus \leaves \}, \quad \onelbset \define \{\child_1 (\othernode) : \ \othernode \in \nodes \setminus \leaves \}
\end{equation}
partition~$\nodes \setminus \{\rootn\}$ according to the value to which~$\decvar_{\binindex(\node)}$ is fixed at node~$\node$. For any node~$\node \in \nodes$, we denote by~$\predecessors (\rootn, \node)$ the set of nodes on the path from~$\rootn$ to~$\node$, including~$\rootn$ and~$\node$, by~$\zerobranchset_\node \subseteq \binaryset$ and~$\onebranchset_\node \subseteq \binaryset$ the indices of the binary variables fixed to zero and to one along this path, and by~$\leaves_\node \subseteq \leaves$ the set of leaf nodes of the subtree rooted at~$\node$.

\begin{table}[b]
  \centering
  \caption{Notation related to the \BB tree~$\tree = (\nodes, \edges)$.}
  \label{tab:tree-notation}
  \small
  \begin{tabular}{@{}l l@{}}
    \toprule
    Symbol & Meaning \\
    \midrule
    $\rootn$, $\leaves$ & Root node and set of leaf nodes \\
    $\parent(\node)$, $\sibling(\node)$ & Parent and sibling of~$\node \in \nodes \setminus \{\rootn\}$ \\
    $\child_0(\node)$, $\child_1(\node)$ & Children of~$\node \in \nodes \setminus \leaves$, with the branching variable fixed to~$0$ and to~$1$ \\
    $\binindex(\node)$ & Index of the variable branched on at~$\parent(\node)$, for~$\node \in \nodes \setminus \{\rootn\}$ \\
    $\zerolbset$, $\onelbset$ & Sets of the children~$\child_0(\othernode)$ and~$\child_1(\othernode)$, $\othernode \in \nodes \setminus \leaves$ \\
    $\predecessors (\rootn, \node)$ & Set of nodes on the path from~$\rootn$ to~$\node$, including~$\rootn$ and~$\node$ \\
    $\zerobranchset_\node$, $\onebranchset_\node$ & Indices of the variables fixed to~$0$ and to~$1$ on the path from~$\rootn$ to~$\node$ \\
    $\leaves_\node$ & Set of leaf nodes of the subtree rooted at~$\node$ \\
    $\lpval_\node$ & Dual bound of node~$\node$ \\
    \bottomrule
  \end{tabular}
\end{table}

As mentioned earlier, we assume that a valid dual bound~$\lpval_\node$ is available at each node~$\node \in \nodes$. This value is a lower bound on the optimal objective of the subproblem at node $\node$:
\begin{equation*}
  \lpval_\node \leq \min \left\{ \objvect^\top \decvar : \
    \begin{array}{l}
      \decvar \in \feasset \\
      \decvar_\binindex = 0, \quad \binindex \in \zerobranchset_\node, \\
      \decvar_\binindex = 1, \quad \binindex \in \onebranchset_\node,
    \end{array}
  \right\}.
\end{equation*}

  We further assume that the dual bounds are monotone along the tree, \ie,~$\lpval_{\parent (\node)} \leq \lpval_\node$ for all~$\node \in \nodes \setminus \{\rootn\}$. This is the case in practice, since the subproblem of a node is a restriction of the subproblem of its parent. Otherwise, replacing~$\lpval_\node$ by~$\max \{\lpval_\node, \lpval_{\parent (\node)}\}$ in a top-down pass over the tree yields monotone dual bounds that remain valid. A leaf node added to complete the tree, as well as a leaf node whose subproblem is infeasible, is assigned the dual bound of its parent node. This bound is valid and preserves the monotonicity of the dual bounds.

  Finally, the \BB tree under consideration need not certify optimality. All results below apply to partial trees, such as those obtained by truncating a \BB tree or by interrupting the \BB algorithm before termination.

\begin{remark}
  \label{rem:infeasible-nodes}
  The convention above does not exploit all the information carried by a leaf node whose subproblem is infeasible. Assigning the dual bound~$+\infty$ to such a node removes its atom from the disjunctive approximation of this section and fixes its auxiliary variable to zero in the formulations of Sections~\ref{sec:mixing-set-outer-approximation} and~\ref{sec:valid-inequalities}, which yields tighter approximations. The results of this paper extend to this setting with minor changes, which mainly concern the separation problems and the star tree inequalities; for instance, the minimum-cost flow problem of Section~\ref{sec:cut-generation-2} may then become infeasible. To keep the exposition simple, we do not pursue this extension.
\end{remark}

\subsection{Definition and Validity of the Outer Approximation}

We construct the outer approximation using the disjunctive theory established by~\textcite{balas2018disjunctive}. Central to this construction are the \emph{atoms} of the branch-and-bound tree, a concept recently employed in~\cite{dey2024theoretical,dey2023lower}. For each leaf node~$\node \in \leaves$, we define the corresponding atom~$\atom_\node$ as
\begin{equation}
  \label{eq:atom-def}
  \atom_\node \define \left\{ \decvar \in \decvarsetLP :
    \begin{array}{ll}
      \objvect^\top \decvar \geq \lpval_\node, \\
      \decvar_{\binindex} = 0, & \binindex \in \zerobranchset_{\node}, \\
      \decvar_{\binindex} = 1, & \binindex \in \onebranchset_{\node}
    \end{array}
    \right\}.
\end{equation}
Each atom~$\atom_\node$ is a polyhedron that incorporates three sets of constraints: the branching decisions that define the path to node~$\node$, the local dual bound~$\objvect^\top \decvar \geq \lpval_\node$ at that node, and the original variable bounds contained within the definition of~$\decvarsetLP$. We recall that~$\decvarsetLP$ is the continuous relaxation of the set~$\decvarset$.

Using this notation, the first outer approximation~$\disjouterapprox$ is defined as the convex hull of the union of all atoms:
\begin{equation}
    \label{eq:balas_outer}
    \disjouterapprox \define \mathrm{conv} \left( \bigcup_{\node \in \leaves} \atom_{\node} \right).
\end{equation}
This formulation differs from the classical approach in disjunctive programming~\cite{balas2018disjunctive}, where the sets in the union are typically defined by the original problem constraints~$\consmat \decvar \geq \conslhs$ rather than by the dual bound inequality~$\objvect^\top \decvar \geq \lpval_\node$. This choice is motivated by our objective to construct approximations using only the information contained within the \BB tree. The dual bound $\lpval_\node$ is a fundamental piece of data generated during the BB process for each node. This decision also results in a more compact formulation. The trade-off between the relative strength of the approximations, measured by inclusion, and the complexity of their separation problems is studied throughout this work.

We now formally establish that~$\disjouterapprox$ is a valid outer approximation of the feasible set~$\feasset$.

\begin{proposition}
  \label{prop:balas-valid}
  The disjunctive approximation is an outer approximation of the set~$\feasset$; that is,
  \begin{equation*}
    \feasset \subseteq \disjouterapprox.
  \end{equation*}
\end{proposition}
\begin{proof}
  Consider an arbitrary point~$\appdecvar \in \feasset$. We show that~$\appdecvar \in \disjouterapprox$. Since by assumption (see Section~\ref{sec:notation}), we assume that the \BB tree partitions the feasible region,~$\appdecvar$ must satisfy the branching conditions corresponding to a unique leaf node~$\othernode \in \leaves$. Therefore, we have
  \begin{equation*}
    \appdecvar \in \left\{ \decvar \in \decvarsetLP :
    \begin{array}{ll}
      \decvar_{\binindex} = 0, & \binindex \in \zerobranchset_{\othernode}, \\
      \decvar_{\binindex} = 1, & \binindex \in \onebranchset_{\othernode}
    \end{array}
    \right\}.
  \end{equation*}
  Furthermore, because~$\appdecvar \in \feasset$, it must also satisfy the local dual bound at node~$\othernode$, so~$\objvect^\top \appdecvar \geq \lpval_\othernode$. It follows that~$\appdecvar \in \atom_\othernode$. Consequently,
  \begin{equation*}
    \appdecvar \in \atom_\othernode \subseteq \bigcup_{\node \in \leaves} \atom_\node \subseteq \mathrm{conv} \left( \bigcup_{\node \in \leaves} \atom_\node \right) = \disjouterapprox.
  \end{equation*}
\end{proof}

Since each set~$\atom_\node$ in the union shares the same recession cone, their convex hull~$\disjouterapprox$ is a polyhedron. This set admits a linear reformulation in an extended variable space with~$\card{\leaves}(\decvarsize + 1)$ auxiliary variables~\cite{balas1985disjunctive,conforti2014integer}. We write each atom as~$\atom_\node = \{ \decvar \in \reals^\decvarsize : \atommat_\node \decvar \geq \atomrhs_\node \}$, where the matrix~$\atommat_\node \in \reals^{\atomconssize \times \decvarsize}$ and vector~$\atomrhs_\node \in \reals^\atomconssize$ are formed by combining the dual bound constraint, the branching constraints (written as inequalities), and the variable bounds from~$\decvarsetLP$. Note that the number of rows,~$\atomconssize$, is the same for all atoms and is given by~$\atomconssize = 1 + 2\card{\binaryset} + o$, where~$o$ is the number of bound constraints on the continuous variables.

\subsection{\texorpdfstring{Separation over~$\disjouterapprox$}{Separation over O1}}

\label{sec:cut-separation}

Since the extended formulation of~$\disjouterapprox$ has~$\card{\leaves}(\decvarsize + 1)$ auxiliary variables, we generate valid inequalities for~$\disjouterapprox$ in the original space of~$\decvar$-variables. We consider the Cut-Generating Linear Program (CGLP) to perform this projection, following the approach of~\textcite{balas1996mixed,balas1993lift}.

For a point~$\appdecvar \in \reals^\decvarsize$, the \CGLP finds a valid inequality~$\cglpmult^\top \decvar \geq \cglpmult_0$ that is most violated by~$\appdecvar$. It associates a vector of nonnegative multipliers~$\cglpvar^\node \in \reals^{\atomconssize}$ with the constraints~$\atommat_\node \decvar \geq \atomrhs_\node$ of each atom, and reads
\begin{equation}
  \label{optim:separation-lp}
  \begin{aligned}
    \max_{\cglpmult, \cglpmult_0, \cglpvar} \quad & \cglpmult_0 - \cglpmult^\top \appdecvar \\
    \st \quad & \cglpmult = \atommat_\node^\top \cglpvar^\node, \quad \node \in \leaves, \\
    & \cglpmult_0 \leq \atomrhs_\node^\top \cglpvar^\node, \quad \node \in \leaves, \\
    & \sum_{\node \in \leaves} \sum_{j = 1}^{\atomconssize} \cglpvar^\node_j = 1, \\
    & \cglpvar^\node \geq 0, \quad \node \in \leaves.
  \end{aligned}
\end{equation}
The constraints associated with a leaf node~$\node$ imply that the inequality~$\cglpmult^\top \decvar \geq \cglpmult_0$ is valid for the atom~$\atom_\node$, and by Farkas' lemma, every inequality that is valid for a nonempty atom admits such multipliers. The \CGLP therefore optimizes over the inequalities that are valid for~$\disjouterapprox$~\cite{balas2018disjunctive}. The constraint on the sum of the multipliers is a normalization that ensures that the \CGLP is bounded~\cite{balas1993lift,balas1996mixed}. This leads to the following observation on the size of the \CGLP.

\begin{observation}
  \label{rem:size-balas}
  The \CGLP~\eqref{optim:separation-lp} has~$\order (\card{\leaves} (\card{\binaryset} + o) + \decvarsize)$ variables and~$\order (\card{\leaves} \decvarsize)$ constraints.
\end{observation}

By Observation~\ref{rem:size-balas}, generating each inequality requires solving a linear program that is about~$\card{\leaves}$ times larger than the original problem, as observed in, \eg,~\cite{chen2012computational,chen2011finite}. Other encodings of this separation problem may be substantially smaller, such as the point-ray formulation of~\textcite{balas2026vpolyhedral}. The next two sections introduce two looser approximations whose separation problems are flow problems on a network whose size only depends on the number of nodes of the \BB tree.

 \section{A Network-Flow-based Outer Approximation}
\label{sec:mixing-set-outer-approximation}

This section presents the second outer approximation. It is obtained by viewing the \BB tree as a network, in which one unit of flow is sent from the root node to the leaf nodes along arcs whose capacities are given by the branching decisions. We first define a mixed-binary set~$\mixedbinset$ from this network and the approximation~$\OtwoApp$, and we show that~$\disjouterapprox = \proj_\decvar (\conv (\mixedbinset)) \subseteq \OtwoApp$. We then show that separating over~$\OtwoApp$ amounts to solving a minimum-cost flow problem on this network.

\subsection{Definition and Validity of the Outer Approximation}
\label{sec:def-flow}

We orient every edge of the \BB tree~$\tree$ away from the root node~$\rootn$, so that each edge becomes an arc~$(\parent (\node), \node)$ for some node~$\node \in \nodes \setminus \{\rootn\}$. Every non-root node~$\node$ thus has exactly one entering arc, and we index the capacity and the flow of this arc by the node~$\node$. For any node~$\node \in \nodes \setminus \{\rootn\}$ and any~$\decvar \in \decvarsetLP$, the capacity~$\capa_\node (\decvar)$ of the arc entering~$\node$ is defined as
\begin{equation}
  \label{eq:capacity}
  \capa_\node (\decvar) \define
  \begin{cases}
    1 - \decvar_{\binindex (\node)}, & \node \in \zerolbset, \\
    \decvar_{\binindex (\node)}, & \node \in \onelbset.
  \end{cases}
\end{equation}
Each capacity~$\capa_\node$ is an affine function of~$\decvar$. Since~$\decvar_\binindex \in [0,1]$ for all~$\binindex \in \binaryset$, and since the two branching decisions taken at a node are complementary, every point~$\decvar \in \decvarsetLP$ satisfies
\begin{equation}
  \label{eq:capacity-props}
  0 \leq \capa_\node (\decvar) \leq 1, \quad \node \in \nodes \setminus \{\rootn\}, \qquad \capa_{\child_0 (\othernode)} (\decvar) + \capa_{\child_1 (\othernode)} (\decvar) = 1, \quad \othernode \in \nodes \setminus \leaves.
\end{equation}
In particular, for~$\decvar \in \decvarset$, the capacity~$\capa_\node (\decvar)$ equals one if~$\decvar$ satisfies the branching decision of the arc entering~$\node$, and zero otherwise.

We associate with every node~$\node \in \nodes$ a variable~$\indvar_\node$, which represents the amount of flow entering the node~$\node$. The flow set~$\flowset$ collects the pairs~$(\decvar, \indvar)$ such that~$\indvar$ sends one unit of flow from the root node to the leaf nodes without exceeding the capacities~$\capa_\node (\decvar)$, $\node \in \nodes \setminus \{\rootn\}$. It reads
\begin{equation}
  \label{eq:Btight}
  \flowset \define \left \{ (\decvar, \indvar) \in \decvarsetLP \times \{0, 1\}^{\card{\nodes}} : \
    \begin{aligned}
      \indvar_\rootn &= 1, \\
      \indvar_\node &= \indvar_{\child_0 (\node)} + \indvar_{\child_1 (\node)}, \quad \node \in \nodes \setminus \leaves, \\
      \indvar_\node &\leq \capa_\node (\decvar), \quad \node \in \nodes \setminus \{\rootn\}
    \end{aligned}
  \right \}.
\end{equation}
Contrary to~$\disjouterapprox$, the second outer approximation is directly defined with auxiliary variables, and these variables are binary on purpose. The variables~$\decvar$, in contrast, are continuous, since they are only required to lie in~$\decvarsetLP$. We denote by~$\flowsetLP$ the continuous relaxation of~$\flowset$, in which~$\indvar \in [0,1]^{\card{\nodes}}$.

A point~$\decvar \in \decvarset$ satisfies the branching decisions of exactly one leaf node. Its capacities are zero or one, so that the only flow~$\indvar$ with~$(\decvar, \indvar) \in \flowsetLP$ sends the unit of flow along the path from the root node to this leaf node. For a fractional point~$\decvar$, the flow may split at a node between its two children, in proportions bounded by the capacities. Figure~\ref{fig:network-tree} in Appendix~\ref{app:network} displays this network for a tree of depth two.

We now define the second set used to construct the network-flow-based outer approximation. It is denoted by~$\msset$ and reads
\begin{equation}
  \label{eq:set-M}
  \msset \define \bigg \{(\decvar, \indvar ) \in \decvarsetLP \times \{0, 1\}^{\card{\nodes}}:
  \objvect^\top \decvar \geq \sum_{\node \in \leaves} \lpval_{\node} \indvar_{\node} \bigg \}.
\end{equation}
Its defining inequality relates the objective function value to the dual bounds of the leaf nodes, weighted by the flow that reaches them. We define the mixed-binary set
\begin{equation}
  \label{eq:def-Q}
  \mixedbinset \define \msset \cap \flowset.
\end{equation}
Only the variables~$\indvar$ of~$\mixedbinset$ are binary. The network-flow-based outer approximation is the projection of the continuous relaxation of~$\mixedbinset$ onto the~$\decvar$-space, that is,
\begin{equation*}
  \OtwoApp \define \proj_\decvar (\HNhat).
\end{equation*}
In the remainder of the paper, we work with~$\OtwoApp$ rather than with the projection of~$\mixedbinset$ itself.

For any node~$\node \in \nodes$, we denote by~$\pathflow^\node \in \{0, 1\}^{\card{\nodes}}$ the vector that sends one unit of flow from the root node to the node~$\node$ along the path~$\predecessors (\rootn, \node)$. Its entries are~$\pathflow^\node_\othernode = 1$ if~$\othernode \in \predecessors (\rootn, \node)$, and~$\pathflow^\node_\othernode = 0$ otherwise.

\begin{remark}
    \label{rem:sumone}
    The flow conservation constraints of~$\flowset$ imply the following. For any node~$\node \in \nodes \setminus \leaves$, we have
    \begin{equation*}
        \indvar_\node = \sum_{\othernode \in \leaves_\node} \indvar_\othernode.
    \end{equation*}
    In particular,~$\sum_{\node \in \leaves} \indvar_\node = \indvar_\rootn = 1$. As a result, all intermediate variables~$\indvar_\node$ for~$\node \in \nodes \setminus \leaves$ can be eliminated from the formulation of~$\flowset$ by substitution.
\end{remark}

  We now show that~$\OtwoApp$ is an outer approximation of~$\feasset$, by relating it to the disjunctive approximation~$\disjouterapprox$. In~$\mixedbinset$, the binary vector~$\indvar$ selects a single path from the root node to a leaf node, and the capacities force~$\decvar$ to satisfy the branching decisions along this path. As a consequence, the projection of the convex hull of~$\mixedbinset$ is~$\disjouterapprox$.

\begin{proposition}
  \label{prop:balas-in-tree}
  The disjunctive approximation is the projection of the convex hull of~$\mixedbinset$, and it is contained in the network-flow-based approximation; that is,
  \begin{equation*}
    \disjouterapprox = \proj_\decvar (\conv (\mixedbinset))
    \subseteq
    \OtwoApp.
  \end{equation*}
\end{proposition}

\begin{proof}
  We proceed in three steps. We first characterize the set~$\mixedbinset$. We then derive the equality~$\disjouterapprox = \proj_\decvar (\conv (\mixedbinset))$. Finally, we derive the inclusion~$\disjouterapprox \subseteq \OtwoApp$.

  \emph{Step 1: characterization of~$\mixedbinset$.} We show that
  \begin{equation}
    \label{eq:char-Q}
    \mixedbinset = \bigcup_{\node \in \leaves} \atom_\node \times \{ \pathflow^{\node} \},
  \end{equation}
  where~$\atom_\node$ is the atom of the leaf node~$\node$ defined in Equation~\eqref{eq:atom-def}, and~$\pathflow^\node$ is the vector that sends one unit of flow from the root node to the node~$\node$, defined in Section~\ref{sec:def-flow}. We prove both inclusions.

  ($\subseteq$) Let~$(\decvar, \indvar) \in \mixedbinset$. Since~$\indvar$ is binary, the constraint~$\indvar_\rootn = 1$ and the flow conservation constraints of~$\flowset$ imply that~$\indvar$ sends one unit of flow along a single path from the root node to a leaf node~$\node$, that is,~$\indvar = \pathflow^\node$. For every node~$\othernode \in \predecessors (\rootn, \node) \setminus \{\rootn\}$, Equation~\eqref{eq:capacity-props} gives~$1 = \indvar_\othernode \leq \capa_\othernode (\decvar) \leq 1$, so that~$\decvar$ satisfies the branching decision of the arc entering~$\othernode$. Moreover, the inequality defining~$\msset$ reduces to~$\objvect^\top \decvar \geq \lpval_\node$. By Definition~\eqref{eq:atom-def},~$\decvar \in \atom_\node$.

  ($\supseteq$) Let~$\node \in \leaves$ and~$\decvar \in \atom_\node$. We show that~$(\decvar, \pathflow^\node) \in \mixedbinset$. The vector~$\pathflow^\node$ satisfies~$\pathflow^\node_\rootn = 1$ and the flow conservation constraints of~$\flowset$. It also satisfies the capacity constraints, since~$\capa_\othernode (\decvar) = 1$ for every node~$\othernode \in \predecessors (\rootn, \node) \setminus \{\rootn\}$, because~$\decvar$ satisfies the branching decisions of~$\node$, and since~$\pathflow^\node_\othernode = 0$ for the other nodes. Hence, since~$\atom_\node \subseteq \decvarsetLP$, we have~$(\decvar, \pathflow^\node) \in \flowset$. Moreover,~$\sum_{\othernode \in \leaves} \lpval_\othernode \pathflow^\node_\othernode = \lpval_\node \leq \objvect^\top \decvar$, so that~$(\decvar, \pathflow^\node) \in \msset$.

  \emph{Step 2: equality.} A point belongs to~$\proj_\decvar (\conv (\mixedbinset))$ if and only if it is a convex combination of the~$\decvar$-components of points of~$\mixedbinset$. By Equation~\eqref{eq:char-Q}, these~$\decvar$-components are the points of the atoms. By Definition~\eqref{eq:balas_outer}, we thus have
  \begin{equation}
    \label{eq:O1-proj}
    \disjouterapprox = \conv \Big( \bigcup_{\node \in \leaves} \atom_\node \Big) = \proj_\decvar (\conv (\mixedbinset)).
  \end{equation}

  \emph{Step 3: inclusion.} We show that~$\disjouterapprox \subseteq \OtwoApp$. The set~$\HNhat$ is convex and contains~$\mixedbinset$, so that~$\conv (\mixedbinset) \subseteq \HNhat$. Projecting onto the~$\decvar$-space and using Equation~\eqref{eq:O1-proj} yields the desired inclusion~$\disjouterapprox \subseteq \proj_\decvar (\HNhat) = \OtwoApp$.
\end{proof}

Propositions~\ref{prop:balas-valid} and~\ref{prop:balas-in-tree} yield~$\feasset \subseteq \disjouterapprox \subseteq \OtwoApp$, so that~$\OtwoApp$ is an outer approximation of~$\feasset$.

\subsection{\texorpdfstring{Separation over~$\OtwoApp$}{Separation over O2}}
\label{sec:cut-generation-2}

We now show that separating a point~$\appdecvar \in \decvarsetLP$ from~$\OtwoApp$ amounts to solving a minimum-cost flow problem on the network of Section~\ref{sec:def-flow}. We first formulate the separation problem as a linear program and relate its feasible set to the set~$\flowsetLP$ (Lemma~\ref{lem:sep-tree}). We then use the extreme points of the feasible region of its dual, which does not depend on~$\appdecvar$, to describe~$\OtwoApp$ by linear inequalities in the~$\decvar$-space (Proposition~\ref{prop:sep-tree}). Finally, we show that the separation problem is solved by a combinatorial algorithm in strongly polynomial time (Corollary~\ref{cor:sep-tree}).

Let~$\appdecvar \in \decvarsetLP$ be the point to be separated. We consider the following linear program. It sends one unit of flow from the root node to the leaf nodes without exceeding the capacities~$\capa_\node (\appdecvar)$. Its objective function~$\sum_{\node \in \leaves} \lpval_\node \indvar_\node$ weights the flow~$\indvar_\node$ entering each leaf node~$\node$ by the dual bound~$\lpval_\node$ of this node. The dual variables associated with each family of constraints are given in parentheses.
\begin{equation}
  \label{eq:sep-lp-tree}
  \begin{aligned}
    \sepvaltree (\appdecvar) \define \min_\indvar \quad & \sum_{\node \in \leaves} \lpval_\node \indvar_\node & \\
    \st \quad & \indvar_\rootn = 1, & (\dualroot) \\
    & \indvar_{\child_0 (\node)} + \indvar_{\child_1 (\node)} - \indvar_\node = 0, \quad \node \in \nodes \setminus \leaves, & (\dualflow_\node) \\
    & \indvar_\node \leq \capa_\node (\appdecvar), \quad \node \in \nodes \setminus \{\rootn\}, & (\dualcapa_\node) \\
    & \indvar_\node \geq 0, \quad \node \in \nodes. &
  \end{aligned}
\end{equation}
Its dual reads
\begin{equation}
  \label{eq:sep-dual-tree}
  \begin{aligned}
    \max_{\dualroot, \dualflow, \dualcapa} \quad & \dualroot - \sum_{\node \in \nodes \setminus \{\rootn\}} \capa_\node (\appdecvar) \dualcapa_\node \\
    \st \quad & \dualroot - \dualflow_\rootn \leq 0, \\
    & \dualflow_{\parent (\node)} - \dualflow_\node \leq \dualcapa_\node, \quad \node \in \nodes \setminus (\leaves \cup \{\rootn\}), \\
    & \dualflow_{\parent (\node)} - \lpval_\node \leq \dualcapa_\node, \quad \node \in \leaves, \\
    & \dualcapa_\node \geq 0, \quad \node \in \nodes \setminus \{\rootn\}.
  \end{aligned}
\end{equation}
Problem~\eqref{eq:sep-lp-tree} is a minimum-cost flow problem in which the point~$\appdecvar$ only appears in the right-hand sides of the capacity constraints. Consequently,~$\appdecvar$ only appears in the objective function of Problem~\eqref{eq:sep-dual-tree}. We denote by~$\dualsettree$ the feasible set of Problem~\eqref{eq:sep-dual-tree}. The constraints defining~$\dualsettree$ only involve the topology of the \BB tree and the dual bounds of its leaf nodes. In particular, the set~$\dualsettree$ does not depend on the point~$\appdecvar$. The following lemma relates Problem~\eqref{eq:sep-lp-tree} to the set~$\flowsetLP$.

\begin{lemma}
  \label{lem:sep-tree}
  Let~$\appdecvar \in \decvarsetLP$. The following statements hold.
  \begin{enumerate}[label=(\roman*)]
    \item A vector~$\indvar$ is feasible for Problem~\eqref{eq:sep-lp-tree} if and only if~$(\appdecvar, \indvar) \in \flowsetLP$.
    \item Problem~\eqref{eq:sep-lp-tree} admits an optimal solution.
  \end{enumerate}
\end{lemma}

\begin{proof}
  \emph{Statement~(i).} When~$\decvar$ is fixed to~$\appdecvar$, the constraints of Problem~\eqref{eq:sep-lp-tree} are the defining constraints of~$\flowsetLP$, except for the bounds~$\indvar \leq 1$. It therefore suffices to show that every feasible vector~$\indvar$ of Problem~\eqref{eq:sep-lp-tree} satisfies~$\indvar \leq 1$. Since~$\indvar \geq 0$, the flow conservation constraints give
  \begin{equation*}
    \indvar_\node = \sum_{\othernode \in \leaves_\node} \indvar_\othernode \leq \sum_{\othernode \in \leaves} \indvar_\othernode = \indvar_\rootn = 1, \quad \node \in \nodes
  \end{equation*}
  (see Remark~\ref{rem:sumone}).

  \emph{Statement~(ii).} We show that the feasible set of Problem~\eqref{eq:sep-lp-tree} is nonempty and bounded. It is bounded, since every feasible vector satisfies~$0 \leq \indvar \leq 1$ (see the proof of Statement~(i)).

  It is nonempty. Indeed, consider the vector that splits the unit of flow entering each non-leaf node between its two children in proportion to their capacities, that is,
  \begin{equation*}
    \indvar_\node \define \prod_{\othernode \in \predecessors (\rootn, \node) \setminus \{\rootn\}} \capa_\othernode (\appdecvar), \quad \node \in \nodes.
  \end{equation*}
  This vector is nonnegative by Equation~\eqref{eq:capacity-props}, and~$\indvar_\rootn = 1$ since the product is empty for~$\node = \rootn$. It satisfies the flow conservation constraints and the capacity constraints, since
  \begin{alignat*}{2}
    \indvar_{\child_0 (\node)} + \indvar_{\child_1 (\node)} &= \indvar_\node \big( \capa_{\child_0 (\node)} (\appdecvar) + \capa_{\child_1 (\node)} (\appdecvar) \big) = \indvar_\node, &\qquad& \node \in \nodes \setminus \leaves, \\
    \indvar_\node &= \indvar_{\parent (\node)} \capa_\node (\appdecvar) \leq \capa_\node (\appdecvar), && \node \in \nodes \setminus \{\rootn\},
  \end{alignat*}
  where the inequality uses~$\indvar_{\parent (\node)} \leq 1$, which holds by Equation~\eqref{eq:capacity-props}.

  A linear program whose feasible set is nonempty and bounded admits an optimal solution, which yields Statement~(ii).
\end{proof}

The following proposition characterizes~$\OtwoApp$ through the optimal value of Problem~\eqref{eq:sep-lp-tree} and through the extreme points of~$\dualsettree$.

\begin{proposition}
  \label{prop:sep-tree}
  Let~$(\dualroot^k, \dualflow^k, \dualcapa^k)$, $k \in \{1, \dots, K\}$, denote the extreme points of~$\dualsettree$. The following statements hold.
  \begin{enumerate}[label=(\roman*)]
    \item A point~$\appdecvar \in \decvarsetLP$ belongs to~$\OtwoApp$ if and only if~$\objvect^\top \appdecvar \geq \sepvaltree (\appdecvar)$.
    \item A point~$\appdecvar \in \decvarsetLP$ belongs to~$\OtwoApp$ if and only if~$\appdecvar$ satisfies the inequalities
    \begin{equation}
      \label{eq:cut-tree}
      \objvect^\top \decvar \geq \dualroot^k - \sum_{\node \in \nodes \setminus \{\rootn\}} \dualcapa^k_\node \capa_\node (\decvar), \quad k \in \{1, \dots, K\}.
    \end{equation}
    In particular, the set~$\OtwoApp$ is a polyhedron.
  \end{enumerate}
\end{proposition}

\begin{proof}
  \emph{Statement~(i).} We prove both directions of the equivalence.

  ($\Rightarrow$) Let~$\appdecvar \in \OtwoApp$. We show that~$\objvect^\top \appdecvar \geq \sepvaltree (\appdecvar)$. By definition of~$\OtwoApp$, there exists a vector~$\indvar$ such that~$(\appdecvar, \indvar) \in \HNhat$. By Lemma~\ref{lem:sep-tree}~(i), the vector~$\indvar$ is feasible for Problem~\eqref{eq:sep-lp-tree}. Hence,
  \begin{equation*}
    \objvect^\top \appdecvar \geq \sum_{\node \in \leaves} \lpval_\node \indvar_\node \geq \sepvaltree (\appdecvar),
  \end{equation*}
  where the first inequality is the inequality defining~$\PTwoMsetLP$, and the second holds because~$\indvar$ is feasible for Problem~\eqref{eq:sep-lp-tree}.

  ($\Leftarrow$) Let~$\appdecvar \in \decvarsetLP$ be such that~$\objvect^\top \appdecvar \geq \sepvaltree (\appdecvar)$ holds. We show that~$\appdecvar \in \OtwoApp$. By Lemma~\ref{lem:sep-tree}, Problem~\eqref{eq:sep-lp-tree} admits an optimal solution~$\indvar^\star$, and~$(\appdecvar, \indvar^\star) \in \flowsetLP$. Moreover,
  \begin{equation*}
    \objvect^\top \appdecvar \geq \sepvaltree (\appdecvar) = \sum_{\node \in \leaves} \lpval_\node \indvar^\star_\node,
  \end{equation*}
  so that~$(\appdecvar, \indvar^\star) \in \PTwoMsetLP$. Hence,~$(\appdecvar, \indvar^\star) \in \HNhat$, which yields~$\appdecvar \in \OtwoApp$.

  \emph{Statement~(ii).} Let~$\appdecvar \in \decvarsetLP$. We first express~$\sepvaltree (\appdecvar)$ through the extreme points of~$\dualsettree$, and we then combine this expression with Statement~(i).

  By Lemma~\ref{lem:sep-tree}~(ii), Problem~\eqref{eq:sep-lp-tree} admits an optimal solution. By strong linear programming duality, Problem~\eqref{eq:sep-dual-tree} therefore has the finite optimal value~$\sepvaltree (\appdecvar)$.

  The polyhedron~$\dualsettree$ contains no line. Indeed, the direction~$(\tilde{\dualroot}, \tilde{\dualflow}, \tilde{\dualcapa})$ of a line contained in~$\dualsettree$ satisfies each constraint of Problem~\eqref{eq:sep-dual-tree} with a zero right-hand side and with equality. Hence,~$\tilde{\dualcapa} = 0$, then~$\tilde{\dualflow} = 0$ from the leaf nodes up to the root node, and finally~$\tilde{\dualroot} = 0$.

  A linear program whose feasible set contains no line and whose optimal value is finite admits an optimal solution at an extreme point of its feasible set~\cite{conforti2014integer}. Recall that~$(\dualroot^k, \dualflow^k, \dualcapa^k)$, $k \in \{1, \dots, K\}$, denote the extreme points of~$\dualsettree$. Hence,
  \begin{equation*}
    \sepvaltree (\appdecvar) = \max_{k \in \{1, \dots, K\}} \left \{ \dualroot^k - \sum_{\node \in \nodes \setminus \{\rootn\}} \dualcapa^k_\node \capa_\node (\appdecvar) \right \}.
  \end{equation*}

  By Statement~(i), the point~$\appdecvar$ belongs to~$\OtwoApp$ if and only if~$\objvect^\top \appdecvar \geq \sepvaltree (\appdecvar)$. By the equality above, this holds if and only if~$\appdecvar$ satisfies all inequalities of~\eqref{eq:cut-tree}.

  Finally, each inequality of~\eqref{eq:cut-tree} is linear in~$\decvar$, because the capacities~$\capa_\node (\decvar)$ are affine functions of~$\decvar$ (see Equation~\eqref{eq:capacity}). Since the number~$K$ of extreme points is finite, the set~$\OtwoApp$ is the intersection of~$\decvarsetLP$ with finitely many halfspaces, and is therefore a polyhedron.
\end{proof}

By Statement~(i) of Proposition~\ref{prop:sep-tree} and weak duality, every feasible solution of Problem~\eqref{eq:sep-dual-tree}, not only an extreme point, yields an inequality of the form~\eqref{eq:cut-tree} that is valid for~$\OtwoApp$. If~$\appdecvar \notin \OtwoApp$, an optimal solution of Problem~\eqref{eq:sep-dual-tree} yields such an inequality that is violated by~$\appdecvar$. The number~$K$ of extreme points of~$\dualsettree$ grows quickly with the size of the \BB tree, but these extreme points need not be enumerated, as the next result shows.

\begin{corollary}
  \label{cor:sep-tree}
  The separation problem over~$\OtwoApp$ is solved by a combinatorial algorithm in strongly polynomial time.
\end{corollary}

\begin{proof}
  Problem~\eqref{eq:sep-lp-tree} is a minimum-cost flow problem on the following network, displayed in Figure~\ref{fig:network-tree} in Appendix~\ref{app:network}. Its nodes are the nodes of the \BB tree and an additional sink node~$t$. Its arcs, with their capacities and objective coefficients, are
  \begin{equation*}
    \begin{alignedat}{2}
      & (\parent (\node), \node), \quad \node \in \nodes \setminus \{\rootn\}: && \qquad \text{capacity~$\capa_\node (\appdecvar)$, objective coefficient~$0$}, \\
      & (\node, t), \quad \node \in \leaves: && \qquad \text{uncapacitated, objective coefficient~$\lpval_\node$}.
    \end{alignedat}
  \end{equation*}
  One unit of flow must be sent from the root node~$\rootn$ to the sink node~$t$, and the flow on the arc entering a node~$\node$ is the variable~$\indvar_\node$. This network has~$\card{\nodes} + 1$ nodes and~$\card{\nodes} - 1 + \card{\leaves}$ arcs, and its capacities are nonnegative by Equation~\eqref{eq:capacity-props}.

  Minimum-cost flow problems with nonnegative capacities are solved in strongly polynomial time, for instance by the enhanced capacity scaling algorithm~\cite[Chapter~10]{ahuja1993network}. In addition to an optimal flow, this algorithm returns optimal node potentials. By complementary slackness, these potentials give an optimal solution of the dual Problem~\eqref{eq:sep-dual-tree}. This solution provides the separating inequality of the form~\eqref{eq:cut-tree} whenever~$\appdecvar \notin \OtwoApp$. Since this network is series-parallel, specialized algorithms can reduce the running time further~\cite{bein1985minimum,booth1993finding}.
\end{proof}

\section{A Mixing-Set-based Outer Approximation}
\label{sec:valid-inequalities}

This section introduces a third outer approximation of the set~$\feasset$. We first give a second formulation of the mixed-binary set~$\mixedbinset$, based on a mixing set. The mixing-set-based outer approximation~$\OthreeApp$ is the projection onto the~$\decvar$-space of the continuous relaxation of this formulation. We then show that~$\OthreeApp$ contains~$\OtwoApp$, so that~$\OthreeApp$ is an outer approximation of~$\feasset$. Finally, we show that separating over~$\OthreeApp$ amounts to solving a maximum flow problem on a network obtained from the \BB tree.

\subsection{Definition and Validity of the Outer Approximation}
Let $ \minlp = \min_{\node \in \leaves} (\lpval_\node)$ be the minimum of the dual bounds over all leaf nodes of the \BB tree. The value~$\minlp$ is a valid lower bound on the optimal objective value of Problem~\eqref{eq:generic-mip}. We also define the set of nodes~$\posnodes \define \{\node \in \nodes: \lpval_\node > \minlp\}$. We define the mixing set~$\nmsset$ as follows:
\begin{equation}
    \label{eq:mix-set-ref}
    \nmsset \define \bigg \{(\decvar, \indvar ) \in \decvarsetLP \times \{0, 1\}^{\card{\nodes}}: \
    \objvect^\top \decvar + (\minlp - \lpval_\node)\indvar_\node\geq \minlp, \ \node \in \nodes \bigg \}.
\end{equation}
The set describes a \textit{strengthened} mixing set in the sense of~\textcite{luedtke2010integer}. In this set, the objective value~$\objvect^\top \decvar$ plays the role of the continuous variable of the mixing set, and each node of the \BB tree contributes one inequality. \textcite{kilincc2014strong} also use a mixing set, to combine the strong-branching inequalities of several candidate variables.

We now describe~$\nmsset$ through the nodes of~$\posnodes$. For every node~$\node \in \posnodes$ and every~$\decvar \in \decvarsetLP$, we define the mixing capacity
\begin{equation}
  \label{eq:mix-capa}
  \mixcapa_\node (\decvar) \define \frac{\objvect^\top \decvar - \minlp}{\lpval_\node - \minlp}.
\end{equation}
Like the capacities~$\capa_\node$, each mixing capacity~$\mixcapa_\node$ is an affine function of~$\decvar$. Since~$\lpval_\node > \minlp$, the inequality of~$\nmsset$ associated with a node~$\node \in \posnodes$ can be rewritten as
\begin{equation*}
  \indvar_\node \leq \frac{\objvect^\top \decvar - \minlp}{\lpval_\node - \minlp} = \mixcapa_\node (\decvar).
\end{equation*}
The inequality associated with a node~$\node \in \nodes \setminus \posnodes$ is implied by~$\objvect^\top \decvar \geq \minlp$, since~$\lpval_\node \leq \minlp$ and~$\indvar_\node \geq 0$. Finally,~$\objvect^\top \decvar \geq \minlp$ is itself the inequality associated with a leaf node attaining~$\minlp$. Hence, we equivalently have
\begin{equation}
  \label{eq:mix-set-reform}
  \nmsset = \Big\{(\decvar, \indvar) \in \decvarsetLP \times \{0,1\}^{\card{\nodes}} : \objvect^\top \decvar \geq \minlp, \ \indvar_\node \leq \mixcapa_\node (\decvar), \ \node \in \posnodes \Big\},
\end{equation}
and the same description holds for~$\nmssetLP$, with~$\indvar \in [0,1]^{\card{\nodes}}$.

The following proposition shows that the sets~$\msset \cap \flowset$ and~$\flowset \cap \nmsset$ are equal.

\begin{proposition}
  \label{prop:equiv-tree-ms}
  The sets~$\msset \cap \flowset$ and~$\flowset \cap \nmsset$ are equal; that is,
  \begin{equation*}
    \msset \cap \flowset = \flowset \cap \nmsset.
  \end{equation*}
\end{proposition}
\begin{proof}
  Let $(\decvar, \indvar) \in \flowset$. We show that~$(\decvar, \indvar) \in \msset$ if and only if~$(\decvar, \indvar) \in \nmsset$.

  By assumption, the \BB tree yields a disjoint partition of the feasible space (see Section~\ref{sec:notation}). Let~$\othernode \in \leaves$ be the unique leaf node for which~$\indvar_\othernode = 1$. We also have~$\sum_{\node \in \leaves} \indvar_\node = 1$.

First, consider the defining inequality for the set~$\msset$ given in Equation~\eqref{eq:set-M}. Since~$\indvar_{\othernode}=1$ and~$\indvar_{\node}=0$ for all~$\node \in \leaves \setminus \{\othernode\}$, this inequality simplifies to
\begin{equation}
  \label{eq:original-best-bnd}
  \objvect^\top \decvar \geq  \lpval_{\othernode}.
\end{equation}

Second, consider the inequality defining~$\nmsset$ associated with a node~$\node \in \nodes$, given in Equation~\eqref{eq:mix-set-ref}. For~$\node = \othernode$, it is Inequality~\eqref{eq:original-best-bnd}. Otherwise, it reduces to~$\objvect^\top \decvar \geq \minlp$ if~$\indvar_\node = 0$, and to~$\objvect^\top \decvar \geq \lpval_\node$ if~$\indvar_\node = 1$, in which case~$\node$ lies on the path from the root node to~$\othernode$. Both inequalities are dominated by Inequality~\eqref{eq:original-best-bnd}, since~$\lpval_\othernode \geq \minlp$ by definition of~$\minlp$, and~$\lpval_\othernode \geq \lpval_\node$ by monotonicity of the dual bounds.
 \end{proof}

By Proposition~\ref{prop:equiv-tree-ms}, the set~$\flowset \cap \nmsset$ is a second formulation of the mixed-binary set~$\mixedbinset$. The mixing-set-based outer approximation that we consider in this section is the projection of the continuous relaxation of this formulation onto the~$\decvar$-space, that is,
\begin{equation*}
  \OthreeApp \define \proj_\decvar (\NMhat).
\end{equation*}
We now compare the continuous relaxations~$\HNhat$ and~$\NMhat$.

\begin{proposition}
  \label{prop:domi1}
  For any point~$(\decvar,\indvar)$ in the continuous relaxation~$\flowsetLP$, the inequality
  \begin{equation*}
    \objvect^\top \decvar \geq  \sum_{\node \in \leaves} \lpval_\node\indvar_\node
  \end{equation*}
  dominates the inequality
  \begin{equation*}
    \objvect^\top \decvar \geq \minlp + (\lpval_\othernode - \minlp)\indvar_\othernode
  \end{equation*}
  for any~$\othernode \in \nodes$.
\end{proposition}

\begin{proof}
  First, we rewrite the expression~$\sum_{\node \in \leaves} \lpval_\node\indvar_\node$. Using the property~$\sum_{\node \in \leaves} \indvar_\node = 1$ discussed in Remark~\ref{rem:sumone}, we have
  \begin{align}
     \sum_{\node \in \leaves} \lpval_\node\indvar_\node
    &= \sum_{\node \in \leaves} (\minlp + \lpval_\node - \minlp) \indvar_\node \nonumber \\
    &= \minlp \left(\sum_{\node \in \leaves} \indvar_\node\right) + \sum_{\node \in \leaves}(\lpval_\node - \minlp) \indvar_\node \nonumber \\
    &= \minlp  + \sum_{\node \in \leaves}(\lpval_\node - \minlp) \indvar_\node. \label{eq:inter-bnd1}
  \end{align}

  Next, we show that for any node~$\othernode \in \nodes$, the following inequality holds:
  \begin{equation}
    \label{eq:bound-announcement}
    \sum_{\node \in \leaves}(\lpval_\node - \minlp) \indvar_\node \geq (\lpval_\othernode - \minlp)  \indvar_\othernode.
  \end{equation}

  From Remark~\ref{rem:sumone}, we know that~$\indvar_\othernode = \sum_{\node \in \leaves_\othernode} \indvar_\node$. Moreover, the dual bounds are non-decreasing, so~$\lpval_\node \geq \lpval_\othernode$ for all~$\node \in \leaves_\othernode$. We can therefore write
  \begin{align}
    \sum_{\node \in \leaves}(\lpval_\node - \minlp) \indvar_\node
    &\geq \sum_{\node \in \leaves_\othernode}(\lpval_\node - \minlp) \indvar_\node \nonumber \\
    &\geq \sum_{\node \in \leaves_\othernode}(\lpval_\othernode - \minlp) \indvar_\node  \nonumber \\
    &= (\lpval_\othernode - \minlp) \sum_{\node \in \leaves_\othernode} \indvar_\node \nonumber\\
    &= (\lpval_\othernode - \minlp)  \indvar_\othernode. \label{eq:inter-bnd2}
  \end{align}
  Combining the result from Equation~\eqref{eq:inter-bnd1} with Inequality~\eqref{eq:inter-bnd2} yields the desired dominance result.
\end{proof}

\begin{corollary}
  \label{cor:treeinms-lp}
  The network-flow-based approximation is contained in the mixing-set-based approximation; that is,
  \begin{equation*}
    \OtwoApp \subseteq \OthreeApp.
  \end{equation*}
  More precisely,~$\HNhat \subseteq \NMhat$.
\end{corollary}
\begin{proof}
  Let~$(\decvar, \indvar) \in \HNhat$. By definition, $(\decvar, \indvar)$ belongs to~$\flowsetLP$ and satisfies the inequality~$\objvect^\top \decvar \geq \sum_{\node \in \leaves} \lpval_\node\indvar_\node$. Proposition~\ref{prop:domi1} implies that this inequality dominates~$\objvect^\top \decvar \geq \minlp + (\lpval_\othernode - \minlp)\indvar_\othernode$ for any~$\othernode \in \nodes$. Therefore, $(\decvar, \indvar)$ also satisfies all the defining inequalities for~$\nmssetLP$. It follows that~$(\decvar, \indvar) \in \NMhat$. Projecting onto the~$\decvar$-space yields~$\OtwoApp \subseteq \OthreeApp$.
\end{proof}

Propositions~\ref{prop:balas-valid} and~\ref{prop:balas-in-tree} and Corollary~\ref{cor:treeinms-lp} yield~$\feasset \subseteq \disjouterapprox \subseteq \OtwoApp \subseteq \OthreeApp$, so that~$\OthreeApp$ is an outer approximation of~$\feasset$.

\subsection{\texorpdfstring{Separation over~$\OthreeApp$}{Separation over O3}}
\label{sec:cut-generation}

We now show that separating a point~$\appdecvar \in \decvarsetLP$ from~$\OthreeApp$ amounts to solving a maximum flow problem on a network obtained from the \BB tree. We follow the structure of Section~\ref{sec:cut-generation-2}. We first formulate the separation problem as a linear program and relate its feasible set to the set~$\NMhat$ (Lemma~\ref{lem:sep-ms}). We then use the extreme points of the feasible region of its dual, which does not depend on~$\appdecvar$, to describe~$\OthreeApp$ by linear inequalities in the~$\decvar$-space (Proposition~\ref{prop:sep-ms}). Finally, we show that the separation problem is solved by a combinatorial algorithm in strongly polynomial time (Corollary~\ref{cor:sep-ms}).

Let~$\appdecvar \in \decvarsetLP$ be the point to be separated. By Equation~\eqref{eq:mix-set-reform}, the inequality~$\objvect^\top \decvar \geq \minlp$ is valid for~$\OthreeApp$, and we assume that~$\objvect^\top \appdecvar \geq \minlp$ holds. When~$\decvar$ is fixed to~$\appdecvar$, Equation~\eqref{eq:mix-set-reform} gives each node~$\node \in \posnodes$ a second upper bound on its flow, namely the mixing capacity~$\mixcapa_\node (\appdecvar)$, in addition to the capacity~$\capa_\node (\appdecvar)$.

We consider the following linear program. It computes the largest amount of flow that can be sent from the root node to the leaf nodes under both capacities. Its constraint~$\indvar_\rootn \leq \capa_\rootn (\appdecvar)$, where we set~$\capa_\rootn (\appdecvar) \define 1$, replaces the constraint~$\indvar_\rootn = 1$ of~$\flowsetLP$. The dual variables associated with each family of constraints are given in parentheses.
\begin{equation}
  \label{eq:sep-lp-ms}
  \begin{aligned}
    \sepvalms (\appdecvar) \define \max_\indvar \quad & \sum_{\node \in \leaves} \indvar_\node & \\
    \st \quad & \indvar_\node \leq \capa_\node (\appdecvar), \quad \node \in \nodes, & (\dualcapa_\node) \\
    & \indvar_\node \leq \mixcapa_\node (\appdecvar), \quad \node \in \posnodes, & (\dualmix_\node) \\
    & \indvar_\node - \indvar_{\child_0 (\node)} - \indvar_{\child_1 (\node)} = 0, \quad \node \in \nodes \setminus \leaves, & (\dualflow_\node) \\
    & \indvar_\node \geq 0, \quad \node \in \nodes. &
  \end{aligned}
\end{equation}
For a fixed point~$\appdecvar$, the two capacities of a node~$\node \in \posnodes$ could be merged into the single capacity~$\min \{ \capa_\node (\appdecvar), \mixcapa_\node (\appdecvar) \}$. We keep them separate in Problem~\eqref{eq:sep-lp-ms}, because this minimum is not an affine function of~$\appdecvar$, whereas~$\capa_\node$ and~$\mixcapa_\node$ are. This choice makes the inequalities derived below linear in~$\decvar$. Its dual reads
\begin{equation}
  \label{eq:sep-dual-ms}
  \begin{aligned}
    \min_{\dualcapa, \dualmix, \dualflow} \quad & \sum_{\node \in \nodes} \capa_\node (\appdecvar) \dualcapa_\node + \sum_{\node \in \posnodes} \mixcapa_\node (\appdecvar) \dualmix_\node \\
    \st \quad & \dualcapa_\rootn + \dualflow_\rootn \geq 0, \\
    & \dualcapa_\node + \dualmix_\node + \dualflow_\node - \dualflow_{\parent (\node)} \geq 0, \quad \node \in \posnodes \setminus \leaves, \\
    & \dualcapa_\node + \dualflow_\node - \dualflow_{\parent (\node)} \geq 0, \quad \node \in \nodes \setminus (\posnodes \cup \leaves \cup \{\rootn\}), \\
    & \dualcapa_\node + \dualmix_\node - \dualflow_{\parent (\node)} \geq 1, \quad \node \in \posnodes \cap \leaves, \\
    & \dualcapa_\node - \dualflow_{\parent (\node)} \geq 1, \quad \node \in \leaves \setminus \posnodes, \\
    & \dualcapa_\node \geq 0, \quad \node \in \nodes, \qquad \dualmix_\node \geq 0, \quad \node \in \posnodes.
  \end{aligned}
\end{equation}
Problem~\eqref{eq:sep-lp-ms} is a maximum flow problem in which the point~$\appdecvar$ only appears in the right-hand sides of the constraints, through~$\capa_\node (\appdecvar)$ and~$\mixcapa_\node (\appdecvar)$. Consequently,~$\appdecvar$ only appears in the objective function of Problem~\eqref{eq:sep-dual-ms}. We denote by~$\dualsetms$ the feasible set of Problem~\eqref{eq:sep-dual-ms}. The constraints defining~$\dualsetms$ only involve the topology of the \BB tree and the dual bounds of its nodes. In particular, the set~$\dualsetms$ does not depend on the point~$\appdecvar$. The following lemma relates Problem~\eqref{eq:sep-lp-ms} to the set~$\NMhat$.

\begin{lemma}
  \label{lem:sep-ms}
  Let~$\appdecvar \in \decvarsetLP$ be such that~$\objvect^\top \appdecvar \geq \minlp$ holds. The following statements hold.
  \begin{enumerate}[label=(\roman*)]
    \item A vector~$\indvar$ is feasible for Problem~\eqref{eq:sep-lp-ms} and satisfies~$\indvar_\rootn = 1$ if and only if~$(\appdecvar, \indvar) \in \NMhat$.
    \item Problem~\eqref{eq:sep-lp-ms} admits an optimal solution, and its optimal value satisfies~$\sepvalms (\appdecvar) \leq 1$.
  \end{enumerate}
\end{lemma}

\begin{proof}
  \emph{Statement~(i).} Since~$\objvect^\top \appdecvar \geq \minlp$, Equation~\eqref{eq:mix-set-reform} shows that a pair~$(\appdecvar, \indvar)$ with~$\indvar \in [0,1]^{\card{\nodes}}$ belongs to~$\nmssetLP$ if and only if~$\indvar_\node \leq \mixcapa_\node (\appdecvar)$ for every node~$\node \in \posnodes$. These bounds are the constraints of Problem~\eqref{eq:sep-lp-ms} associated with the dual variables~$\dualmix$.

  When~$\decvar$ is fixed to~$\appdecvar$, the other constraints of Problem~\eqref{eq:sep-lp-ms}, together with~$\indvar_\rootn = 1$, are the defining constraints of~$\flowsetLP$, except for the bounds~$\indvar \leq 1$. It therefore suffices to show that every feasible vector~$\indvar$ of Problem~\eqref{eq:sep-lp-ms} satisfies~$\indvar \leq 1$. By Equation~\eqref{eq:capacity-props} and since~$\capa_\rootn (\appdecvar) = 1$, we have~$\indvar_\node \leq \capa_\node (\appdecvar) \leq 1$ for every node~$\node \in \nodes$.

  \emph{Statement~(ii).} We show that the feasible set of Problem~\eqref{eq:sep-lp-ms} is nonempty and bounded, and that the objective value of every feasible vector is at most one. The vector~$\indvar = 0$ is feasible, since~$\capa_\node (\appdecvar) \geq 0$ by Equation~\eqref{eq:capacity-props}, and~$\mixcapa_\node (\appdecvar) \geq 0$ because~$\objvect^\top \appdecvar \geq \minlp$. The feasible set is bounded, since every feasible vector satisfies~$0 \leq \indvar \leq 1$ (see the proof of Statement~(i)). Hence, Problem~\eqref{eq:sep-lp-ms} admits an optimal solution. Moreover, every feasible vector~$\indvar$ satisfies~$\sum_{\node \in \leaves} \indvar_\node = \indvar_\rootn \leq \capa_\rootn (\appdecvar) = 1$ by Remark~\ref{rem:sumone}, so that~$\sepvalms (\appdecvar) \leq 1$.
\end{proof}

The following proposition characterizes~$\OthreeApp$ through the optimal value of Problem~\eqref{eq:sep-lp-ms} and through the extreme points of~$\dualsetms$.

\begin{proposition}
  \label{prop:sep-ms}
  Let~$(\dualcapa^k, \dualmix^k, \dualflow^k)$, $k \in \{1, \dots, K\}$, denote the extreme points of~$\dualsetms$. The following statements hold.
  \begin{enumerate}[label=(\roman*)]
    \item A point~$\appdecvar \in \decvarsetLP$ belongs to~$\OthreeApp$ if and only if~$\objvect^\top \appdecvar \geq \minlp$ and~$\sepvalms (\appdecvar) \geq 1$.
    \item A point~$\appdecvar \in \decvarsetLP$ belongs to~$\OthreeApp$ if and only if~$\appdecvar$ satisfies~$\objvect^\top \decvar \geq \minlp$ and the inequalities
    \begin{equation}
      \label{eq:cut-ms}
      \sum_{\node \in \nodes} \dualcapa^k_\node \capa_\node (\decvar) + \sum_{\node \in \posnodes} \dualmix^k_\node \mixcapa_\node (\decvar) \geq 1, \quad k \in \{1, \dots, K\}.
    \end{equation}
    In particular, the set~$\OthreeApp$ is a polyhedron.
  \end{enumerate}
\end{proposition}

\begin{proof}
  \emph{Statement~(i).} We prove both directions of the equivalence.

  ($\Rightarrow$) Let~$\appdecvar \in \OthreeApp$. We show that~$\objvect^\top \appdecvar \geq \minlp$ and~$\sepvalms (\appdecvar) \geq 1$. By definition of~$\OthreeApp$, there exists a vector~$\indvar$ such that~$(\appdecvar, \indvar) \in \NMhat$. By Equation~\eqref{eq:mix-set-reform},~$\objvect^\top \appdecvar \geq \minlp$. By Lemma~\ref{lem:sep-ms}~(i), the vector~$\indvar$ is feasible for Problem~\eqref{eq:sep-lp-ms} and satisfies~$\indvar_\rootn = 1$. By Remark~\ref{rem:sumone},
  \begin{equation*}
    \sepvalms (\appdecvar) \geq \sum_{\node \in \leaves} \indvar_\node = \indvar_\rootn = 1.
  \end{equation*}

  ($\Leftarrow$) Let~$\appdecvar \in \decvarsetLP$ be such that~$\objvect^\top \appdecvar \geq \minlp$ and~$\sepvalms (\appdecvar) \geq 1$ hold. We show that~$\appdecvar \in \OthreeApp$. By Lemma~\ref{lem:sep-ms}~(ii), Problem~\eqref{eq:sep-lp-ms} admits an optimal solution~$\indvar^\star$, and~$\sepvalms (\appdecvar) \leq 1$. By Remark~\ref{rem:sumone},
  \begin{equation*}
    \indvar^\star_\rootn = \sum_{\node \in \leaves} \indvar^\star_\node = \sepvalms (\appdecvar) = 1.
  \end{equation*}
  By Lemma~\ref{lem:sep-ms}~(i),~$(\appdecvar, \indvar^\star) \in \NMhat$, which yields~$\appdecvar \in \OthreeApp$.

  \emph{Statement~(ii).} Let~$\appdecvar \in \decvarsetLP$. If~$\objvect^\top \appdecvar < \minlp$, then~$\appdecvar \notin \OthreeApp$ by Statement~(i), and~$\appdecvar$ violates~$\objvect^\top \decvar \geq \minlp$, so that both sides of the equivalence fail. We now assume that~$\objvect^\top \appdecvar \geq \minlp$. We first express~$\sepvalms (\appdecvar)$ through the extreme points of~$\dualsetms$, and we then combine this expression with Statement~(i).

  By Lemma~\ref{lem:sep-ms}~(ii), Problem~\eqref{eq:sep-lp-ms} admits an optimal solution. By strong linear programming duality, Problem~\eqref{eq:sep-dual-ms} therefore has the finite optimal value~$\sepvalms (\appdecvar)$.

  The polyhedron~$\dualsetms$ contains no line. Indeed, as in the proof of Proposition~\ref{prop:sep-tree}, the direction~$(\tilde{\dualcapa}, \tilde{\dualmix}, \tilde{\dualflow})$ of a line contained in~$\dualsetms$ satisfies each constraint of Problem~\eqref{eq:sep-dual-ms} with a zero right-hand side and with equality. Hence,~$\tilde{\dualcapa} = 0$ and~$\tilde{\dualmix} = 0$, and then~$\tilde{\dualflow} = 0$ from the root node down to the leaf nodes.

  A linear program whose feasible set contains no line and whose optimal value is finite admits an optimal solution at an extreme point of its feasible set~\cite{conforti2014integer}. Recall that~$(\dualcapa^k, \dualmix^k, \dualflow^k)$, $k \in \{1, \dots, K\}$, denote the extreme points of~$\dualsetms$. Hence,
  \begin{equation*}
    \sepvalms (\appdecvar) = \min_{k \in \{1, \dots, K\}} \left \{ \sum_{\node \in \nodes} \dualcapa^k_\node \capa_\node (\appdecvar) + \sum_{\node \in \posnodes} \dualmix^k_\node \mixcapa_\node (\appdecvar) \right \}.
  \end{equation*}

  By Statement~(i), the point~$\appdecvar$ belongs to~$\OthreeApp$ if and only if~$\sepvalms (\appdecvar) \geq 1$. By the equality above, this holds if and only if~$\appdecvar$ satisfies all inequalities of~\eqref{eq:cut-ms}.

  Finally, each inequality of~\eqref{eq:cut-ms} is linear in~$\decvar$, because~$\capa_\node$ and~$\mixcapa_\node$ are affine functions of~$\decvar$ (see Equations~\eqref{eq:capacity} and~\eqref{eq:mix-capa}). Since the number~$K$ of extreme points is finite, the set~$\OthreeApp$ is the intersection of~$\decvarsetLP$ with finitely many halfspaces, and is therefore a polyhedron.
\end{proof}

As in Section~\ref{sec:cut-generation-2}, Statement~(i) of Proposition~\ref{prop:sep-ms} and weak duality show that every feasible solution of Problem~\eqref{eq:sep-dual-ms} yields an inequality of the form~\eqref{eq:cut-ms} that is valid for~$\OthreeApp$. If~$\appdecvar \notin \OthreeApp$ and~$\objvect^\top \appdecvar \geq \minlp$, an optimal solution of Problem~\eqref{eq:sep-dual-ms} yields such an inequality that is violated by~$\appdecvar$. The extreme points of~$\dualsetms$ need not be enumerated, as the next result shows.

\begin{corollary}
  \label{cor:sep-ms}
  The separation problem over~$\OthreeApp$ is solved by a combinatorial algorithm in strongly polynomial time.
\end{corollary}

\begin{proof}
  If~$\objvect^\top \appdecvar < \minlp$ holds, the inequality~$\objvect^\top \decvar \geq \minlp$ separates the point~$\appdecvar$.

  Otherwise, Problem~\eqref{eq:sep-lp-ms} is a maximum flow problem from a source node~$s$ to the sink node~$t$ on the network of Corollary~\ref{cor:sep-tree}, modified as follows; see Figure~\ref{fig:network-ms} in Appendix~\ref{app:network}. An arc~$(s, \rootn)$ of capacity one is added, which models the constraint~$\indvar_\rootn \leq \capa_\rootn (\appdecvar) = 1$. The capacity of the arc entering a node~$\node \in \posnodes$ becomes~$\min \{ \capa_\node (\appdecvar), \mixcapa_\node (\appdecvar) \}$, and the arcs entering~$t$ have no objective coefficient. The flow on the arc entering a node~$\node \in \nodes$ is the variable~$\indvar_\node$. This network has~$\card{\nodes} + 2$ nodes and~$\card{\nodes} + \card{\leaves}$ arcs, and its capacities are nonnegative, since~$\capa_\node (\appdecvar) \geq 0$ by Equation~\eqref{eq:capacity-props} and~$\mixcapa_\node (\appdecvar) \geq 0$ because~$\objvect^\top \appdecvar \geq \minlp$.

  Maximum flow problems with nonnegative capacities are solved in strongly polynomial time, for instance by the FIFO preflow-push algorithm~\cite[Chapter~7]{ahuja1993network}, which returns a minimum cut. The minimum cut problem is the dual of the maximum flow problem~\cite[Chapter~6]{ahuja1993network}. The minimum cut returned by the algorithm therefore gives an optimal solution of the dual Problem~\eqref{eq:sep-dual-ms}. This solution provides the separating inequality of the form~\eqref{eq:cut-ms} whenever~$\appdecvar \notin \OthreeApp$, that is, whenever the capacity of the cut is smaller than one. Since this network is series-parallel, specialized algorithms can reduce the running time further~\cite{bein1985minimum,booth1993finding}.
\end{proof}

Propositions~\ref{prop:sep-tree} and~\ref{prop:sep-ms} show that neither approximation requires a \CGLP. The next section introduces a family of valid inequalities that is also separated by a combinatorial algorithm in strongly polynomial time.

 \section{The Star Tree Inequalities}
\label{sec:star-ineq-insp}

  This section introduces a family of valid inequalities for~$\feasset$. \textcite{kilincc2014strong} derive valid inequalities from strong branching and strengthen them through a mixing set. The inequalities of this section generalize their construction from strong-branching candidates, which can be viewed as \BB trees of depth one, to \BB trees of arbitrary depth. Contrary to the approximations of Sections~\ref{sec:mixing-set-outer-approximation} and~\ref{sec:valid-inequalities}, their validity relies on the integrality of the variables~$\indvar$ in the mixing set~$\nmsset$.

We call these inequalities star tree inequalities~(\STI s), as their construction combines the star inequalities presented in~\cite{gunluk2001mixing} and the flow conservation and capacity constraints of the flow set~$\flowset$.

This section is organized as follows. First, we review the formulation of star inequalities, also known as mixing inequalities. Second, we introduce the STIs and prove their validity for the set~$\flowset \cap \nmsset$. Subsequently, we analyze the closure of the STIs and establish that it contains the network-flow-based approximation~$\OtwoApp$. Finally, we demonstrate that the associated separation problem can be solved to optimality with a combinatorial algorithm in strongly polynomial time.

The following theorem, adapted from \textcite{luedtke2010integer}, presents the star inequalities for the strengthened mixing set defined in Equation~\eqref{eq:mix-set-ref}.

\begin{theorem}[\textcite{luedtke2010integer}]
    \label{thm:luedtke-mixing}
    Let~$\permu$ be a permutation of the set~$\posnodes$ that satisfies~$\lpval_{{\permu}_1} \geq \dots \geq \lpval_{{\permu}_{\card{\posnodes}}}$. The inequalities
    \begin{equation}
        \label{eq:star-ineq}
        \objvect^\top \decvar \geq \lpval_{t_{k + 1}} + \sum_{j = 1}^{k} (\lpval_{t_j} - \lpval_{t_{j+1}}) \indvar_{t_j} , \quad T = \{t_1, \dots, t_k\} \subseteq \posnodes
    \end{equation}
    with $\lpval_{t_1}  \geq \dots \geq \lpval_{t_{k}}$ and~$\lpval_{t_{k+1}} \define \minlp$, are valid for~$\nmsset$.
\end{theorem}

The star inequalities are restricted to subsets of~$\posnodes$ because, by Equation~\eqref{eq:mix-set-reform}, the inequalities of~$\nmsset$ associated with the other nodes are implied by~$\objvect^\top \decvar \geq \minlp$. Moreover, adding~$\permu_1$ to~$T$ never weakens Inequality~\eqref{eq:star-ineq}, since it adds the nonnegative term~$(\lpval_{\permu_1} - \lpval_{t_1}) \indvar_{\permu_1}$ to its right-hand side. In what follows, we therefore always set~$t_1 \define \permu_1$ and write~$T = \{t_2, \dots, t_k\} \subseteq \posnodes \setminus \{\permu_1\}$.

Star inequalities are studied extensively in the mathematical programming literature due to their favorable computational properties and the frequent appearance of mixing sets in integer programming models. Notably, the separation problem for the star inequalities~\eqref{eq:star-ineq} can be solved in polynomial time by reduction to a shortest-path problem on a graph \cite{atamturk2000mixed, gunluk2001mixing}.

\subsection{Definition and Validity of the \STI}
\label{sec:star-tree-ineq}

We now formally introduce the \STI s. Their construction relies on the following lemma, which establishes a lower bound on the auxiliary variables~$\indvar_\node$.

\begin{lemma}
\label{lem:validity-lb-indvar}
For any point $(\decvar, \indvar) \in \flowsetLP$ and any node~$\node \in \nodes$, the following inequality holds:
\begin{equation}
\label{eq:reform-lb-indvar}
\indvar_\node \geq 1 - \sum_{\binindex \in \onebranchset_\node} (1 - \decvar_\binindex) - \sum_{\binindex \in \zerobranchset_\node} \decvar_\binindex.
\end{equation}
\end{lemma}

The proof, by induction on the nodes of the \BB tree, is given in Appendix~\ref{app:proofs}.

The following proposition formally defines the \STI s. Their validity is established by strengthening the star inequalities from Theorem~\ref{thm:luedtke-mixing} with the lower bound from Lemma~\ref{lem:validity-lb-indvar}.

\begin{proposition}[Star tree inequalities]
    \label{prop:star-tree}
    Let~$\permu$ be a permutation of the set~$\posnodes$ that satisfies~$\lpval_{{\permu}_1} \geq \dots \geq \lpval_{{\permu}_{\card{\posnodes}}}$. For any set $T = \{t_2, \dots, t_k\} \subseteq \posnodes \setminus \{\permu_1\}$ with $t_1 \define \permu_1$, $\lpval_{t_2} \geq \dots \geq \lpval_{t_{k}}$, and $\lpval_{t_{k+1}} \define \minlp$, the inequality
    \begin{equation}
        \label{eq:new-star-ineq}
        \objvect^\top \decvar \geq \lpval_{t_1} - \sum_{j = 1}^{k} (\lpval_{t_j} - \lpval_{t_{j+1}}) \viparam_{t_j} (\decvar),
      \end{equation}
      where for any~$\node \in \nodes$:
      \begin{equation*}
        \viparam_\node (\decvar) \define \min \bigg ( 1, \sum_{\binindex \in \onebranchset_{\node}} (1 - \decvar_i) + \sum_{\binindex \in \zerobranchset_{\node}} \decvar_i \bigg ),
      \end{equation*}
    is valid for the set~$\flowset \cap \nmsset$.
\end{proposition}

\begin{proof}
  From Theorem~\ref{thm:luedtke-mixing}, the star inequality
  \begin{equation}
    \label{eq:star-1}
    \objvect^\top \decvar \geq \lpval_{t_{k+1}} + \sum_{j = 1}^{k} (\lpval_{t_j} - \lpval_{t_{j+1}}) \indvar_{t_j}
  \end{equation}
  is valid for the set~$\nmsset$. Consequently, it is also valid for the set~$\flowset \cap \nmsset$. By Lemma~\ref{lem:validity-lb-indvar} and since~$\indvar_\node \geq 0$, the inequality
  \begin{equation}
    \indvar_\node \geq \max \bigg(0, 1 - \sum_{\binindex \in \onebranchset_\node} (1 - \decvar_\binindex)  - \sum_{\binindex \in \zerobranchset_\node} \decvar_\binindex\bigg ) = 1 - \viparam_\node (\decvar)
    \label{eq:star-2}
  \end{equation}
  is valid for the set~$\flowset$, and hence for~$\flowset \cap \nmsset$, for every node~$\node \in \nodes$. Since the coefficients~$(\lpval_{t_j} - \lpval_{t_{j+1}})$ are nonnegative by construction, substituting Inequality~\eqref{eq:star-2} into Inequality~\eqref{eq:star-1} yields
  \begin{equation*}
    \objvect^\top \decvar \geq \lpval_{t_{k+1}} + \sum_{j = 1}^{k} (\lpval_{t_j} - \lpval_{t_{j+1}}) \big( 1 - \viparam_{t_j} (\decvar) \big) = \lpval_{t_1} - \sum_{j = 1}^{k} (\lpval_{t_j} - \lpval_{t_{j+1}}) \viparam_{t_j} (\decvar),
  \end{equation*}
  where the equality holds because the sum~$\sum_{j = 1}^{k} (\lpval_{t_j} - \lpval_{t_{j+1}})$ telescopes to~$\lpval_{t_1} - \lpval_{t_{k+1}}$.
\end{proof}

\begin{corollary}
  Inequality~\eqref{eq:new-star-ineq} is valid for~$\flowset \cap \nmsset$ and depends only on the~$\decvar$ variables. Thus, \STI s are valid for~$\proj_\decvar(\flowset \cap \nmsset)$.
\end{corollary}

The \STI s generalize the strong-branching cuts of~\textcite{kilincc2014strong}. Consider a \BB tree of depth one that branches on a variable~$\decvar_\binindex$, with children~$\node^0$ and~$\node^1$ in which~$\decvar_\binindex$ is fixed to zero and to one, and assume, possibly after complementing~$\decvar_\binindex$, that~$\lpval_{\node^0} < \lpval_{\node^1}$. Then~$\minlp = \lpval_{\node^0}$, $\posnodes = \{\node^1\}$, and~$\viparam_{\node^1} (\decvar) = 1 - \decvar_\binindex$, so that the \STI\xspace associated with~$T = \emptyset$ reads
\begin{equation*}
  \objvect^\top \decvar \geq \lpval_{\node^1} - (\lpval_{\node^1} - \lpval_{\node^0})(1 - \decvar_\binindex) = \lpval_{\node^0} + (\lpval_{\node^1} - \lpval_{\node^0}) \decvar_\binindex.
\end{equation*}
This inequality is the strong-branching cut of~\textcite{kilincc2014strong}.

\begin{remark}
  \label{rem:nonlinear}
  Although Inequality~\eqref{eq:new-star-ineq} is nonlinear because of the~$\viparam$ operator, it implies a family of linear inequalities. For any given point~$\decvar$, a linear inequality is obtained by replacing each~$\viparam$ operator with either~$1$ or the corresponding sum over variables. The choice that yields the strongest cut is determined during separation, as formalized in Proposition~\ref{prop:separation}.
\end{remark}

The following result establishes a relationship between the continuous relaxation~$\HNhat$ and the closure of the STIs, which we define subsequently.

\begin{proposition}
  \label{prop:dominance-star-classic}
  Let~$\permu$ be a permutation of the set~$\posnodes$ that satisfies~$\lpval_{{\permu}_1} \geq \dots \geq \lpval_{{\permu}_{\card{\posnodes}}}$. For any point~$(\decvar,\indvar)$ in the continuous relaxation~$\flowsetLP$ and any~$T = \{t_2, \dots, t_k\} \subseteq \posnodes \setminus \{\permu_1\}$ with~$\lpval_{t_2} \geq \dots \geq \lpval_{t_k}$, the inequality
  \begin{equation*}
    \objvect^\top \decvar \geq  \sum_{\node \in \leaves} \lpval_\node\indvar_\node
  \end{equation*}
  dominates the inequality
  \begin{equation*}
    \objvect^\top \decvar \geq \lpval_{t_{k + 1}} + \sum_{j = 1}^{k} (\lpval_{t_j} - \lpval_{t_{j+1}}) \indvar_{t_j},
  \end{equation*}
where~$t_{1} \define \permu_{1}$ and~$\lpval_{t_{k + 1}} \define \minlp$.
\end{proposition}

\begin{proof}
  From Equation~\eqref{eq:inter-bnd1} in the proof of Proposition~\ref{prop:domi1} we know that since~$(\decvar,\indvar) \in \flowsetLP$,
  \begin{equation}
    \label{eq:repeat-eq}
    \sum_{\node \in \leaves} \lpval_\node \indvar_\node = \minlp + \sum_{\node \in \leaves} (\lpval_\node  - \minlp) \indvar_\node
  \end{equation}
  holds. It remains to show that for any $T = \{t_2, \dots, t_k\} \subseteq \posnodes \setminus \{\permu_1\}$, the inequality
  \begin{equation}
    \label{ew:to-show-star}
     \lpval_{t_{k + 1}} + \sum_{j = 1}^{k} (\lpval_{t_j} - \lpval_{t_{j+1}}) \indvar_{t_j} \leq \minlp + \sum_{\node \in \leaves} (\lpval_\node - \minlp) \indvar_\node
   \end{equation}
   holds. By Remark~\ref{rem:sumone}, we have~$\indvar_{t_j} = \sum_{\node \in \leaves_{t_j}} \indvar_\node$ for every~$j$. Since~$\lpval_{t_{k+1}} = \minlp$, the left-hand side of Inequality~\eqref{ew:to-show-star} is therefore equal to
   \begin{equation*}
     \minlp + \sum_{\node \in \leaves} \indvar_\node \sum_{j \in \{1, \dots, k\} : \, \node \in \leaves_{t_j}} (\lpval_{t_j} - \lpval_{t_{j+1}}).
   \end{equation*}
   Let~$\node \in \leaves$ be a leaf node that descends from at least one node of~$\{t_1, \dots, t_k\}$, and let~$j^*$ be the smallest index such that~$\node \in \leaves_{t_{j^*}}$. Since the differences~$\lpval_{t_j} - \lpval_{t_{j+1}}$ are nonnegative, the inner sum is at most
   \begin{equation*}
     \sum_{j = j^*}^{k} (\lpval_{t_j} - \lpval_{t_{j+1}}) = \lpval_{t_{j^*}} - \minlp \leq \lpval_\node - \minlp,
   \end{equation*}
   where the equality holds because the sum telescopes, and the inequality holds because~$\node$ descends from~$t_{j^*}$ and the dual bounds are monotone. For the other leaf nodes, the inner sum is zero. Since~$\indvar \geq 0$, this yields Inequality~\eqref{ew:to-show-star}.
\end{proof}

The dominance result in Proposition~\ref{prop:dominance-star-classic} allows us to compare the \STI\xspace closure with the network-flow-based approximation~$\OtwoApp$. With some abuse of notation, let~$\treeinf$ denote the information collected from the \BB tree. The \STI\xspace closure for~$\treeinf$ is defined as
\begin{equation*}
  \clos (\treeinf) \define \bigcap_{T \subseteq \posnodes \setminus \{\permu_1\}} \left \{ \decvar \in \reals^\decvarsize: \objvect^\top \decvar \geq \lpval_{t_1} - \sum_{j = 1}^{k} (\lpval_{t_j} - \lpval_{t_{j+1}}) \viparam_{t_j} (\decvar) \right \},
\end{equation*}
where the intersection is over all valid sets $T = \{t_2, \dots, t_k\}$. By Remark~\ref{rem:nonlinear}, the right-hand side of an \STI\xspace is the maximum of the right-hand sides of finitely many linear inequalities, so that the \STI\xspace holds if and only if these linear inequalities hold. Hence,~$\clos (\treeinf)$ is a polyhedron. We can now state the following inclusion result.

\begin{corollary}
  \label{cor:treeinms}
  The network-flow-based approximation is contained in the STI closure; that is,
  \begin{equation*}
    \OtwoApp \subseteq \clos (\treeinf).
  \end{equation*}
\end{corollary}
\begin{proof}
  By Proposition~\ref{prop:dominance-star-classic}, every point~$(\decvar, \indvar) \in \HNhat$ satisfies the star inequality~\eqref{eq:star-1} for every~$T \subseteq \posnodes \setminus \{\permu_1\}$. By Lemma~\ref{lem:validity-lb-indvar}, which is stated for the continuous relaxation~$\flowsetLP$, it also satisfies Inequality~\eqref{eq:star-2}. Combining both inequalities as in the proof of Proposition~\ref{prop:star-tree} shows that every~$\decvar \in \OtwoApp$ satisfies every~\STI\xspace and hence belongs to~$\clos (\treeinf)$.
\end{proof}

Corollary~\ref{cor:treeinms} shows that no \STI\xspace cuts off a point of~$\OtwoApp$, so that the \STI s do not provide a stronger approximation, and separating over~$\OtwoApp$ is also possible in strongly polynomial time (Corollary~\ref{cor:sep-tree}). Their interest lies in the link that they establish with the strong-branching inequalities of~\textcite{kilincc2014strong} and with the mixing-set literature.

\subsection{\texorpdfstring{Separation over~$\clos (\treeinf)$}{Separation over C(G)}}
\label{sec:cut-gen-ms}

We now study the separation problem for the \STI s. As for the approximations of Sections~\ref{sec:mixing-set-outer-approximation} and~\ref{sec:valid-inequalities}, we do not use a \CGLP. The following proposition shows that separation can be done with an efficient combinatorial algorithm.

\begin{proposition}
  \label{prop:separation}
  The separation problem over the closure~$\clos (\treeinf)$ of the \STI s is solved by a combinatorial algorithm in strongly polynomial time.
\end{proposition}
\begin{proof}
 The separation problem for a given point~$\appdecvar$ requires us to find a sequence~$T=(t_2, \dots, t_k)$ that defines the most violated \STI.  This is equivalent to maximizing the right-hand side of Inequality~\eqref{eq:new-star-ineq}. 

  Given a point~$\appdecvar$, the values~$\viparam_\node (\appdecvar)$ are computed in~$\order (\card{\nodes})$ operations by a pass from the root node to the leaf nodes, since~$\viparam_\rootn (\appdecvar) = 0$ and
  \begin{equation*}
    \viparam_\node (\appdecvar) = \min \big( 1, \viparam_{\parent (\node)} (\appdecvar) + 1 - \capa_\node (\appdecvar) \big), \quad \node \in \nodes \setminus \{\rootn\}.
  \end{equation*}
  Each of these values also determines whether the minimum operator is replaced by~$1$ or by the sum in the linear inequality of Remark~\ref{rem:nonlinear}. The separation problem then reduces to finding a sequence~$T = (t_2, \dots, t_k)$ that minimizes the sum~$\sum_{j=1}^{k} (\lpval_{t_j} - \lpval_{t_{j+1}}) \viparam_{t_j}(\appdecvar)$. As in the proof of Proposition~\ref{prop:star-tree}, this problem is the separation problem of the star inequalities~\eqref{eq:star-1} at the point~$\appdecvar$ with~$\indvar_\node = 1 - \viparam_\node (\appdecvar)$, $\node \in \posnodes$. This separation problem is solved in strongly polynomial time, for instance by the greedy algorithm of~\textcite[Remark~3]{kilincc2021joint}. The resulting sequence provides an \STI\xspace that is most violated by~$\appdecvar$.
\end{proof}

  Propositions~\ref{prop:balas-valid} and~\ref{prop:balas-in-tree}, Corollary~\ref{cor:treeinms-lp}, and Corollary~\ref{cor:treeinms} yield the chain of inclusions
  \begin{equation}
    \label{eq:inclusion-chain}
    \feasset \subseteq \disjouterapprox = \proj_\decvar (\conv (\mixedbinset)) \subseteq \OtwoApp \subseteq \OthreeApp \cap \clos (\treeinf),
  \end{equation}
  which is summarized in Figure~\ref{fig:hierarchy}. In particular, every \STI\xspace is valid for~$\feasset$.

 \section{Strictness of the Hierarchy}
\label{sec:overview}
\label{sec:example}

This section shows that every inclusion of~\eqref{eq:inclusion-chain} can be strict, and that~$\OthreeApp$ and~$\clos (\treeinf)$ are not comparable. We first show that, when the objective function is a single continuous variable without bounds, each of the four approximations is described by a function of the binary variables. We then introduce an example and compare these functions at two points. Finally, we confirm the resulting relations on two sections of the example, and we collect them in Theorem~\ref{thm:hierarchy}.

The function that describes~$\disjouterapprox$ relies on the set
\begin{equation}
  \label{eq:def-S}
  S \define \bigcup_{\node \in \leaves} \left\{ y \in \{0,1\}^{\binvarsize} :
    \begin{array}{ll}
      y_{\binindex} = 0, & \binindex \in \zerobranchset_{\node}, \\
      y_{\binindex} = 1, & \binindex \in \onebranchset_{\node}
    \end{array}
  \right\} \times \{ \lpval_\node \}.
\end{equation}
By the assumptions of Section~\ref{sec:notation}, each vector~$y \in \{0,1\}^{\binvarsize}$ satisfies the branching decisions of exactly one leaf node~$\node$. Hence,~$S$ consists of the points~$(y, \lpval_\node)$, one for each such vector~$y$.

\begin{proposition}
  \label{prop:value-functions}
  Assume that~$\objvect^\top \decvar = \decvar_\decvarsize$, where~$\decvar_\decvarsize$ is a continuous variable without bounds, and denote by~$\decvar_{\binaryset} \define (\decvar_\binindex)_{\binindex \in \binaryset}$ the vector of binary variables of a point~$\decvar \in \decvarsetLP$. A point~$\decvar \in \decvarsetLP$ belongs to~$\disjouterapprox$, $\OtwoApp$, $\OthreeApp$, or~$\clos (\treeinf)$ if and only if~$\decvar_\decvarsize \geq g (\decvar_{\binaryset})$, where~$g$ is the function~$g_1$, $g_2$, $g_3$, or~$g_\clos$, respectively. These four functions are defined on~$[0,1]^{\binvarsize}$ by
  \begin{align*}
    g_1 (\decvar_{\binaryset}) &= \min_\tau \big\{ \tau : (\decvar_{\binaryset}, \tau) \in \conv (S) \big\}, \\
    g_2 (\decvar_{\binaryset}) &= \min_\indvar \bigg\{ \sum_{\node \in \leaves} \lpval_\node \indvar_\node : (\decvar, \indvar) \in \flowsetLP \bigg\}, \\
    g_3 (\decvar_{\binaryset}) &= \min_\indvar \Big\{ \max_{\node \in \nodes} \big( \minlp + (\lpval_\node - \minlp) \indvar_\node \big) : (\decvar, \indvar) \in \flowsetLP \Big\}, \\
    g_\clos (\decvar_{\binaryset}) &= \max_{T \subseteq \posnodes \setminus \{\permu_1\}} \bigg\{ \lpval_{t_1} - \sum_{j = 1}^{k} (\lpval_{t_j} - \lpval_{t_{j+1}}) \viparam_{t_j} (\decvar) \bigg\}.
  \end{align*}
\end{proposition}

\begin{proof}
  Let~$\decvar \in \decvarsetLP$. The set~$\flowsetLP$ only depends on~$\decvar$ through~$\decvarsetLP$ and through the capacities~$\capa_\node (\decvar)$, which only involve binary variables by Definition~\eqref{eq:capacity}. The same holds for the functions~$\viparam_\node (\decvar)$ of Proposition~\ref{prop:star-tree}. Hence, the four functions only depend on~$\decvar_{\binaryset}$. We now consider each set in turn.

  \emph{Set~$\disjouterapprox$.} Since~$\decvar_\decvarsize$ has no bounds, Definition~\eqref{eq:atom-def} shows that an atom~$\atom_\node$ restricts~$\decvar_{\binaryset}$ to the face of~$[0,1]^{\binvarsize}$ defined by the branching decisions of~$\node$, restricts~$\decvar_\decvarsize$ to~$[\lpval_\node, + \infty)$, and restricts the other continuous variables to their bounds. This face is the convex hull of the binary vectors~$y$ that satisfy the branching decisions of~$\node$. Since all atoms impose the same bounds on the other continuous variables, $\decvar$ belongs to~$\disjouterapprox = \conv (\bigcup_{\node \in \leaves} \atom_\node)$ if and only if
  \begin{equation*}
    (\decvar_{\binaryset}, \decvar_\decvarsize) \in \conv (S) + \big( \{0\} \times \reals_+ \big),
  \end{equation*}
  that is, if and only if~$\decvar_\decvarsize \geq g_1 (\decvar_{\binaryset})$.

  \emph{Set~$\OtwoApp$.} By Statement~(i) of Proposition~\ref{prop:sep-tree},~$\decvar \in \OtwoApp$ if and only if~$\decvar_\decvarsize \geq \sepvaltree (\decvar)$. By Statement~(i) of Lemma~\ref{lem:sep-tree},~$\sepvaltree (\decvar) = g_2 (\decvar_{\binaryset})$.

  \emph{Set~$\OthreeApp$.} By Definition~\eqref{eq:mix-set-ref}, and since~$\objvect^\top \decvar = \decvar_\decvarsize$,~$\decvar \in \OthreeApp$ if and only if there exists a vector~$\indvar$ such that~$(\decvar, \indvar) \in \flowsetLP$ and~$\decvar_\decvarsize \geq \minlp + (\lpval_\node - \minlp) \indvar_\node$ for every node~$\node \in \nodes$. Such a vector exists if and only if~$\decvar_\decvarsize \geq g_3 (\decvar_{\binaryset})$, because the minimum defining~$g_3 (\decvar_{\binaryset})$ is attained over the nonempty polytope~$\{\indvar : (\decvar, \indvar) \in \flowsetLP\}$ (see Lemma~\ref{lem:sep-tree}).

  \emph{Set~$\clos (\treeinf)$.} The equivalence follows from the definition of~$\clos (\treeinf)$, since~$\objvect^\top \decvar = \decvar_\decvarsize$.
\end{proof}

As a consequence of Proposition~\ref{prop:value-functions} and of the chain~\eqref{eq:inclusion-chain}, the four functions satisfy
\begin{equation}
  \label{eq:g-chain}
  g_1 \geq g_2 \geq \max \{ g_3, g_\clos \} \quad \text{on~$[0,1]^{\binvarsize}$}.
\end{equation}
Moreover, if an inequality of~\eqref{eq:g-chain} is strict at some point, the corresponding inclusion of~\eqref{eq:inclusion-chain} is strict. Similarly, if~$g_3 - g_\clos$ takes both positive and negative values, then~$\OthreeApp$ and~$\clos (\treeinf)$ are not comparable. We now introduce the example. For this example, we will observe that~$g_1 > g_2 > \max \{ g_3, g_\clos \}$ at some point, and that~$g_3 - g_\clos$ takes both positive and negative values.

\begin{example}
  \label{ex:strict}
  Let~$\decvarsize = 3$,~$\decvar_1, \decvar_2 \in \{0,1\}$,~$\decvar_3 \in \reals$,~$\objvect^\top \decvar = \decvar_3$, and
  \begin{equation*}
    \feasset = \left \{\decvar \in \decvarset : \decvar_3 \geq 1 + \decvar_1 + \decvar_2 \right \}.
  \end{equation*}
  Consider the \BB tree of Figure~\ref{fig:ex-tree}, whose dual bounds are valid and monotone. Then~$\minlp = \min_{\node \in \leaves} \lpval_\node = 1$ and~$\posnodes = \{a_1, b_0, b_1\}$.
\end{example}

\begin{figure}[h]
  \centering
  \begin{tikzpicture}[
      innernode/.style = {draw, circle, minimum size = 7mm, inner sep = 0pt},
      leafnode/.style = {draw, rounded corners = 1.5pt, minimum size = 7mm, inner sep = 2pt},
      branchlabel/.style = {midway, fill = white, inner sep = 1.5pt, font = \footnotesize},
      boundlabel/.style = {font = \footnotesize},
      level 1/.style = {sibling distance = 52mm, level distance = 16mm},
      level 2/.style = {sibling distance = 24mm, level distance = 16mm},
      edge from parent/.style = {draw = black!55, thin, edge from parent path = {(\tikzparentnode) -- (\tikzchildnode)}},
    ]
    \node[innernode, label = {[boundlabel] right:$\lpval_\rootn = 1$}] {$\rootn$}
    child {node[innernode, label = {[boundlabel] left:$\lpval_a = 1$}] {$a$}
      child {node[leafnode, label = {[boundlabel] below:$\lpval_{a_0} = 1$}] {$a_0$}
        edge from parent node[branchlabel, above left] {$\decvar_2 = 0$}}
      child {node[leafnode, label = {[boundlabel] below:$\lpval_{a_1} = 2$}] {$a_1$}
        edge from parent node[branchlabel, above right] {$\decvar_2 = 1$}}
      edge from parent node[branchlabel, above left] {$\decvar_1 = 0$}}
    child {node[innernode, label = {[boundlabel] right:$\lpval_b = 1$}] {$b$}
      child {node[leafnode, label = {[boundlabel] below:$\lpval_{b_0} = 2$}] {$b_0$}
        edge from parent node[branchlabel, above left] {$\decvar_2 = 0$}}
      child {node[leafnode, label = {[boundlabel] below:$\lpval_{b_1} = 3$}] {$b_1$}
        edge from parent node[branchlabel, above right] {$\decvar_2 = 1$}}
      edge from parent node[branchlabel, above right] {$\decvar_1 = 1$}};
  \end{tikzpicture}
  \caption{The \BB tree of Example~\ref{ex:strict}. Circles are internal nodes and rectangles are leaf nodes. The dual bound of a node is given next to the node, and the branching decisions are given along the edges.}
  \label{fig:ex-tree}
\end{figure}

In Example~\ref{ex:strict}, the vector of binary variables is~$\decvar_{\binaryset} = (\decvar_1, \decvar_2)$. Table~\ref{tab:example} reports the values of the functions~$g_1$, $g_2$, $g_3$, and~$g_\clos$ of Proposition~\ref{prop:value-functions} at the points~$p = (1/2, 1/2)$ and~$q = (1, 1/3)$ of the~$(\decvar_1, \decvar_2)$-square~$[0,1]^2$. Appendix~\ref{app:example} computes these values by hand.

\begin{table}[h]
  \centering
  \caption{Values of the functions~$g_1$, $g_2$, $g_3$ and~$g_\clos$ of Example~\ref{ex:strict} at the points~$p$ and~$q$.}
  \label{tab:example}
  \begin{tabular}{@{}l c c c@{}}
    \toprule
    Set & Function & Value at~$p$ & Value at~$q$ \\
    \midrule
    $\disjouterapprox$ & $g_1$ & $2$ & $7/3$ \\
    $\OtwoApp$ & $g_2$ & $3/2$ & $7/3$ \\
    $\OthreeApp$ & $g_3$ & $4/3$ & $5/3$ \\
    $\clos (\treeinf)$ & $g_\clos$ & $1$ & $2$ \\
    \bottomrule
  \end{tabular}
\end{table}

By~\eqref{eq:g-chain} and the observations that follow it, the values of Table~\ref{tab:example} yield the following relations:
\begin{equation*}
  \begin{alignedat}{2}
    & \feasset \subsetneq \disjouterapprox: && \qquad (\tfrac12, \tfrac12, 2) \in \disjouterapprox \setminus \feasset, \ \text{since~$g_1 (p) = 2$ and~$\decvar_1 = \tfrac12$}, \\
    & \disjouterapprox \subsetneq \OtwoApp: && \qquad g_1 (p) = 2 > \tfrac32 = g_2 (p), \\
    & \OtwoApp \subsetneq \OthreeApp \cap \clos (\treeinf): && \qquad g_2 (p) = \tfrac32 > \tfrac43 = \max \{ g_3 (p), g_\clos (p) \}, \\
    & \OthreeApp \not\subseteq \clos (\treeinf): && \qquad g_3 (q) = \tfrac53 < 2 = g_\clos (q), \\
    & \clos (\treeinf) \not\subseteq \OthreeApp: && \qquad g_\clos (p) = 1 < \tfrac43 = g_3 (p).
  \end{alignedat}
\end{equation*}

Figure~\ref{fig:ex-sections} confirms these relations visually. It displays the four functions along two segments of the~$(\decvar_1, \decvar_2)$-square~$[0,1]^2$ that contain~$p$ and~$q$, respectively.\footnote{The curves of Figure~\ref{fig:ex-sections} are piecewise linear on both segments. Their breakpoints were obtained by evaluating each function on a grid of each segment, by solving the linear programs of Proposition~\ref{prop:value-functions} for~$g_1$, $g_2$, and~$g_3$, and by enumerating the \STI s for~$g_\clos$.} The relations above do not rely on this figure, since they only use the values of Table~\ref{tab:example}, which are computed by hand in Appendix~\ref{app:example}.

\begin{figure}[t]
  \centering
  \begin{tikzpicture}
    \begin{groupplot}[
        group style = {group size = 3 by 1, horizontal sep = 12mm},
        width = 4.4cm, height = 4.6cm,
        tick label style = {font = \footnotesize},
        label style = {font = \small},
        title style = {font = \small},
        grid = major, grid style = {gray!20},
      ]
      \nextgroupplot[
        title = {(a) Segments A and B},
        xmin = -0.22, xmax = 1.22, ymin = -0.22, ymax = 1.22,
        xlabel = {$\decvar_1$}, ylabel = {$\decvar_2$},
        xtick = {0, 0.5, 1}, ytick = {0, 0.5, 1},
        axis equal image, grid = none, clip = false]
      \addplot[gray!60, fill = gray!8] coordinates {(0,0) (1,0) (1,1) (0,1) (0,0)};
      \addplot[thick, cQ, densely dashed, ->] coordinates {(0,0) (1,1)};
      \addplot[thick, cQ, densely dashed, ->] coordinates {(1,0) (1,1)};
      \node[font = \footnotesize, fill = gray!8, inner sep = 1pt] at (axis cs:0.36,0.56) {A};
      \node[font = \footnotesize, fill = gray!8, inner sep = 1pt] at (axis cs:0.9,0.56) {B};
      \addplot[only marks, mark = *, mark size = 1.6pt, cQ] coordinates {(0.5,0.5) (1,0.3333)};
      \node[font = \footnotesize, anchor = north west] at (axis cs:0.5,0.5) {$p$};
      \node[font = \footnotesize, anchor = east] at (axis cs:0.98,0.3333) {$q$};
      \node[font = \scriptsize, anchor = north east, inner sep = 1pt] at (axis cs:0,0) {$a_0$};
      \node[font = \scriptsize, anchor = south east, inner sep = 1pt] at (axis cs:0,1) {$a_1$};
      \node[font = \scriptsize, anchor = north west, inner sep = 1pt] at (axis cs:1,0) {$b_0$};
      \node[font = \scriptsize, anchor = south west, inner sep = 1pt] at (axis cs:1,1) {$b_1$};
      \nextgroupplot[
        title = {(b) Section A},
        xmin = -0.05, xmax = 1.05, ymin = 0.8, ymax = 3.6,
        xlabel = {$t$}, ylabel = {$\decvar_3$},
        xtick = {0, 0.5, 1}, ytick = {1, 1.5, 2, 2.5, 3},
        legend to name = sectionlegend, legend columns = 3,
        legend style = {font = \footnotesize, draw = none, column sep = 3mm}]
      \addplot[cO1, very thick] coordinates {(0,1) (1,3)};
      \addlegendentry{$\disjouterapprox$}
      \addplot[cO2, very thick, dashed] coordinates {(0,1) (0.5,1.5) (1,3)};
      \addlegendentry{$\OtwoApp$}
      \addplot[cO3, very thick, dashdotted] coordinates {(0,1) (0.6,1.4) (1,3)};
      \addlegendentry{$\OthreeApp$}
      \addplot[cC, very thick, densely dotted] coordinates {(0,1) (0.5,1) (1,3)};
      \addlegendentry{$\clos (\treeinf)$}
      \addplot[cQ, line width = 2pt] coordinates {(0,1) (0,3.6)};
      \addlegendentry{$\proj_\decvar (\mixedbinset)$}
      \addplot[cQ, line width = 2pt, forget plot] coordinates {(1,3) (1,3.6)};
      \addplot[only marks, mark = *, mark size = 1.8pt, cQ, forget plot] coordinates {(0,1) (1,3)};
      \draw[gray, thin] (axis cs:0.5,0.8) -- (axis cs:0.5,3.6);
      \node[font = \footnotesize, anchor = south west] at (axis cs:0.5,0.8) {$p$};
      \nextgroupplot[
        title = {(c) Section B},
        xmin = -0.05, xmax = 1.05, ymin = 0.8, ymax = 3.6,
        xlabel = {$t$},
        xtick = {0, 0.3333, 1}, xticklabels = {$0$, $\frac13$, $1$}, ytick = {1, 1.5, 2, 2.5, 3}, yticklabels = {}]
      \addplot[cO1, very thick] coordinates {(0,2) (1,3)};
      \addplot[cO2, very thick, dashed] coordinates {(0,2) (1,3)};
      \addplot[cO3, very thick, dashdotted] coordinates {(0,2) (0.3333,1.6667) (1,3)};
      \addplot[cC, very thick, densely dotted] coordinates {(0,2) (0.5,2) (1,3)};
      \addplot[cQ, line width = 2pt] coordinates {(0,2) (0,3.6)};
      \addplot[cQ, line width = 2pt] coordinates {(1,3) (1,3.6)};
      \addplot[only marks, mark = *, mark size = 1.8pt, cQ] coordinates {(0,2) (1,3)};
      \draw[gray, thin] (axis cs:0.3333,0.8) -- (axis cs:0.3333,3.6);
      \node[font = \footnotesize, anchor = south west] at (axis cs:0.3333,0.8) {$q$};
    \end{groupplot}
    \node[anchor = north] at ([yshift = -11mm]group c2r1.south) {\pgfplotslegendfromname{sectionlegend}};
  \end{tikzpicture}
  \caption{The sets of Example~\ref{ex:strict} on two sections. Panel~(a) shows the~$(\decvar_1, \decvar_2)$-square~$[0,1]^2$ and two segments, both parametrized by~$t \in [0,1]$. Segment~A goes from~$a_0$ to~$b_1$ through~$p$, and segment~B goes from~$b_0$ to~$b_1$ through~$q$. Panel~(b) shows the section of each set above segment~A, that is, its points~$(t, t, \decvar_3)$, with~$t$ on the horizontal axis and~$\decvar_3$ on the vertical axis. The point~$p$ corresponds to~$t = 1/2$. Panel~(c) shows, in the same way, the points~$(1, t, \decvar_3)$ above segment~B. The point~$q$ corresponds to~$t = 1/3$. In Panels~(b) and~(c), each set consists of the points above its curve. The thick vertical half-lines at~$t = 0$ and~$t = 1$ form the section of~$\proj_\decvar (\mixedbinset)$, that is, the atoms of the two leaf nodes located at the end points of the segment. In Panel~(c), the curves of~$\disjouterapprox$ and~$\OtwoApp$ coincide. Both sections confirm the values of Table~\ref{tab:example}.}
  \label{fig:ex-sections}
\end{figure}

The chain~\eqref{eq:inclusion-chain} and the relations derived from Example~\ref{ex:strict} yield the following theorem.

\begin{theorem}
  \label{thm:hierarchy}
  For every mixed-binary linear program and every \BB tree that satisfy the assumptions of Section~\ref{sec:notation}, the outer approximations satisfy
  \begin{equation*}
    \feasset \subseteq \disjouterapprox = \proj_\decvar (\conv (\mixedbinset)) \subseteq \OtwoApp \subseteq \OthreeApp \cap \clos (\treeinf).
  \end{equation*}
  Moreover, every inclusion of this chain can be strict, and the sets~$\OthreeApp$ and~$\clos (\treeinf)$ are not comparable in general.
\end{theorem}

\begin{remark}
  \label{rem:joint}
  At the point~$p$ of Example~\ref{ex:strict}, the optimal flow of Problem~\eqref{eq:sep-lp-tree} sends~$1/2$ to each of the leaf nodes~$a_0$ and~$b_0$ (see Appendix~\ref{app:example}). Each of these two leaf nodes respects its capacity~$1 - \decvar_2 = 1/2$, but together they receive a flow of one, which exceeds~$1 - \decvar_2$. This suggests imposing the capacities of the nodes that branch on the same variable jointly, that is,
  \begin{equation*}
    \sum_{\node \in \onelbset : \, \binindex (\node) = \binindex} \indvar_\node \leq \decvar_\binindex
    \quad \text{and} \quad
    \sum_{\node \in \zerolbset : \, \binindex (\node) = \binindex} \indvar_\node \leq 1 - \decvar_\binindex,
    \quad \binindex \in \binaryset.
  \end{equation*}
  These inequalities are valid for the set~$\mixedbinset$, since a variable is fixed at most once along any path from the root node to a leaf node. Adding them to~$\HNhat$ therefore yields an outer approximation whose projection lies between~$\disjouterapprox$ and~$\OtwoApp$ (see Proposition~\ref{prop:balas-in-tree}), and in Example~\ref{ex:strict} they raise the value at~$p$ from~$3/2$ to~$2$. The price is the separation problem. The new inequalities couple nodes located in different subtrees, so that fixing~$\decvar$ no longer leaves a network flow problem but a general linear program, whose size still depends only on the tree. We do not pursue this strengthening and leave it for further work.
\end{remark}

 \section{Conclusion}
\label{sec:conclusion}

This paper introduced three polyhedral outer approximations of the feasible region of mixed-binary programs, derived a posteriori from a branch-and-bound~(\BB) tree, together with a new family of valid inequalities, the star tree inequalities. These sets are relaxations of a single mixed-binary set built from the tree, and they form an inclusion hierarchy in which every inclusion can be strict (Theorem~\ref{thm:hierarchy}). Except for the disjunctive approximation, which is separated through a cut-generating linear program, their separation problems are a minimum-cost flow problem, a maximum flow problem, and the separation problem of star inequalities, which are solved by combinatorial algorithms in strongly polynomial time.

Further research avenues include the use of these approximations in settings such as reoptimization, decomposition methods, bilevel optimization, or inverse optimization, and their extension to general integer variables and general disjunctions, to several \BB trees, and to infeasible leaf nodes (see Remark~\ref{rem:infeasible-nodes}). Each of these extensions is a research question in its own right. Whether stronger relaxations, such as the one obtained with joint capacities in Remark~\ref{rem:joint}, admit combinatorial separation algorithms is also an open question.

 \section*{Acknowledgments}
\label{sec:acknowledgements}

The support of SCALE-AI through its Research Chairs program, IVADO, and the Canada First Research Excellence Fund (Apog\'ee/CFREF) is gratefully acknowledged.

\section*{Conflict of interest}

The authors declare that they have no conflict of interest that could have influenced the work reported in this paper.

\appendix
\section{\texorpdfstring{The Networks of Corollaries~\ref{cor:sep-tree} and~\ref{cor:sep-ms}}{The Networks of Corollaries 1 and 3}}
\label{app:network}

This appendix displays, in Figures~\ref{fig:network-tree} and~\ref{fig:network-ms}, the networks constructed in the proofs of Corollaries~\ref{cor:sep-tree} and~\ref{cor:sep-ms} for the same \BB tree of depth two.

\begin{figure}[htbp]
  \centering
  \begin{tikzpicture}[
    >={Stealth[length=2mm]},
    nd/.style={draw, circle, minimum size=7mm, inner sep=0pt},
    lf/.style={draw, rounded corners=1.5pt, minimum size=7mm, inner sep=1.5pt},
    ex/.style={draw, circle, fill=gray!15, minimum size=7mm, inner sep=0pt},
    arc/.style={->, draw=black!75, thin},
    fixed/.style={draw=black!75, double, double distance=1.2pt, -{Implies}},
    lab/.style={font=\footnotesize, inner sep=1.5pt},
    x=0.85mm, y=0.8mm]
    \node[nd] (r) at (45,36) {$\rootn$};
    \node[nd] (a) at (15,22) {$a$};
    \node[nd] (b) at (75,22) {$b$};
    \node[lf] (a0) at (0,4) {$a_0$};
    \node[lf] (a1) at (30,4) {$a_1$};
    \node[lf] (b0) at (60,4) {$b_0$};
    \node[lf] (b1) at (90,4) {$b_1$};
    \node[ex] (t) at (45,-16) {$t$};
    \draw[fixed] (45,50) -- node[lab, right] {$= 1$} (r);
    \draw[fixed] (t) -- node[lab, right] {$= 1$} (45,-30);
    \draw[arc] (r) -- node[lab, above left] {$\capa_a$} (a);
    \draw[arc] (r) -- node[lab, above right] {$\capa_b$} (b);
    \draw[arc] (a) -- node[lab, left] {$\capa_{a_0}$} (a0);
    \draw[arc] (a) -- node[lab, right] {$\capa_{a_1}$} (a1);
    \draw[arc] (b) -- node[lab, left] {$\capa_{b_0}$} (b0);
    \draw[arc] (b) -- node[lab, right] {$\capa_{b_1}$} (b1);
    \draw[arc, densely dashed] (a0) -- node[lab, pos=0.4, below left] {$\lpval_{a_0}$} (t);
    \draw[arc, densely dashed] (a1) -- node[lab, pos=0.4, left, xshift=-1.6mm] {$\lpval_{a_1}$} (t);
    \draw[arc, densely dashed] (b0) -- node[lab, pos=0.4, right, xshift=1.6mm] {$\lpval_{b_0}$} (t);
    \draw[arc, densely dashed] (b1) -- node[lab, pos=0.4, below right] {$\lpval_{b_1}$} (t);
  \end{tikzpicture}
  \caption{The network of Corollary~\ref{cor:sep-tree} for a \BB tree of depth two. The double arrows indicate that exactly one unit of flow enters the network at the root node~$\rootn$ and leaves it at the sink node~$t$. Each arc of the tree is labeled with its capacity, evaluated at the point~$\appdecvar$ to be separated, and has objective coefficient zero. Each dashed arc is uncapacitated and labeled with its objective coefficient. The flow on the arc entering a node~$\node$ is the variable~$\indvar_\node$.}
  \label{fig:network-tree}
\end{figure}

\begin{figure}[htbp]
  \centering
  \begin{tikzpicture}[
    >={Stealth[length=2mm]},
    nd/.style={draw, circle, minimum size=7mm, inner sep=0pt},
    lf/.style={draw, rounded corners=1.5pt, minimum size=7mm, inner sep=1.5pt},
    ex/.style={draw, circle, fill=gray!15, minimum size=7mm, inner sep=0pt},
    arc/.style={->, draw=black!75, thin},
    lab/.style={font=\footnotesize, inner sep=1.5pt},
    x=1mm, y=0.8mm]
    \node[ex] (s) at (45,50) {$s$};
    \node[nd] (r) at (45,36) {$\rootn$};
    \node[nd] (a) at (15,22) {$a$};
    \node[nd] (b) at (75,22) {$b$};
    \node[lf] (a0) at (0,4) {$a_0$};
    \node[lf] (a1) at (30,4) {$a_1$};
    \node[lf] (b0) at (60,4) {$b_0$};
    \node[lf] (b1) at (90,4) {$b_1$};
    \node[ex] (t) at (45,-16) {$t$};
    \draw[arc] (s) -- node[lab, right] {$\leq 1$} (r);
    \draw[arc] (r) -- node[lab, above left] {$\capa_a$} (a);
    \draw[arc] (r) -- node[lab, above right] {$\capa_b$} (b);
    \draw[arc] (a) -- node[lab, left] {$\capa_{a_0}$} (a0);
    \draw[arc, very thick] (a) -- node[lab, right] {$\min \{\capa_{a_1}, \mixcapa_{a_1}\}$} (a1);
    \draw[arc, very thick] (b) -- node[lab, left] {$\min \{\capa_{b_0}, \mixcapa_{b_0}\}$} (b0);
    \draw[arc, very thick] (b) -- node[lab, right] {$\min \{\capa_{b_1}, \mixcapa_{b_1}\}$} (b1);
    \foreach \l in {a0,a1,b0,b1} \draw[arc, densely dashed] (\l) -- (t);
  \end{tikzpicture}
  \caption{The network of Corollary~\ref{cor:sep-ms} for a \BB tree of depth two in which~$\posnodes = \{a_1, b_0, b_1\}$. The flow from the source node~$s$ to the sink node~$t$ is maximized. The arc~$(s, \rootn)$ has capacity one, so that at most one unit of flow is sent, whereas exactly one unit is sent in Figure~\ref{fig:network-tree}. Each other arc is labeled with its capacity, evaluated at the point~$\appdecvar$ to be separated, and dashed arcs are uncapacitated. The flow on the arc entering a node~$\node$, including the arc~$(s, \rootn)$, is the variable~$\indvar_\node$. The thick arcs enter the nodes of~$\posnodes$, and their capacity is the smaller of the capacity and the mixing capacity of the node.}
  \label{fig:network-ms}
\end{figure}

\FloatBarrier
\section{\texorpdfstring{Proof of Lemma~\ref{lem:validity-lb-indvar}}{Proof of Lemma 3}}
\label{app:proofs}

This appendix proves Lemma~\ref{lem:validity-lb-indvar}, which states that every point~$(\decvar, \indvar) \in \flowsetLP$ satisfies Inequality~\eqref{eq:reform-lb-indvar} for every node~$\node \in \nodes$.

\begin{proof}[Proof of Lemma~\ref{lem:validity-lb-indvar}]
  We proceed by induction on the nodes of a \BB tree. The argument only uses the linear constraints defining~$\flowsetLP$ and therefore applies to the continuous relaxation. \\
  \textit{Base Case.} Consider the root node,~$\node = \rootn$. By definition, the sets~$\onebranchset_\rootn$ and~$\zerobranchset_\rootn$ are empty. Inequality~\eqref{eq:reform-lb-indvar} therefore simplifies to~$\indvar_\rootn \geq 1$. This is valid because~$\indvar_\rootn = 1$ is a defining constraint of the set~$\flowsetLP$, as specified in Equation~\eqref{eq:Btight}.\\
  \textit{Inductive step.} Assume the inequality holds for a node~$\parent(\othernode)$. We show it holds for its child~$\othernode \in \nodes \setminus \{\rootn\}$. We assume \Wlog that~$\othernode \in \onelbset$ as defined in Equation~\eqref{eq:def-lbset}. The flow conservation constraints and the capacity constraints of~$\flowsetLP$ give
  \begin{equation*}
    \indvar_\othernode = \indvar_{\parent (\othernode)} - \indvar_{\sibling (\othernode)}
    \qquad \text{and} \qquad
    \indvar_{\sibling (\othernode)} \leq \capa_{\sibling (\othernode)} (\decvar) = 1 - \decvar_{\binindex (\othernode)},
  \end{equation*}
  where the last equality holds because the capacities of the two sibling nodes~$\othernode$ and~$\sibling (\othernode)$ sum to one. Hence, any point of~$\flowsetLP$ satisfies
  \begin{align*}
    \indvar_\othernode
    &\geq 1 - (1 - \indvar_{\parent (\othernode)}) - (1 - \decvar_{\binindex (\othernode )})\\
    &\geq 1 - \sum_{\binindex \in \onebranchset_{\parent (\othernode)}} (1 - \decvar_\binindex) - \sum_{\binindex \in \zerobranchset_{\parent (\othernode)}} \decvar_\binindex - (1 - \decvar_{\binindex (\othernode )}) \\
    &= 1 - \sum_{\binindex \in \onebranchset_\othernode} (1 - \decvar_\binindex) - \sum_{\binindex \in \zerobranchset_\othernode} \decvar_\binindex.
  \end{align*}
 The case~$\othernode \in \zerolbset$ is complementary.\\
\end{proof}

\section{\texorpdfstring{Values of the Functions~$g_1$, $g_2$, $g_3$, and~$g_\clos$ in Example~\ref{ex:strict}}{Values of the Functions in Example 1}}
\label{app:example}

This appendix verifies the values of Table~\ref{tab:example} with Proposition~\ref{prop:value-functions}. In Example~\ref{ex:strict}, we have~$\minlp = 1$ and~$\posnodes = \{a_1, b_0, b_1\}$, and the capacities are~$\capa_a (\decvar) = 1 - \decvar_1$, $\capa_b (\decvar) = \decvar_1$, $\capa_{a_0} (\decvar) = \capa_{b_0} (\decvar) = 1 - \decvar_2$, and~$\capa_{a_1} (\decvar) = \capa_{b_1} (\decvar) = \decvar_2$.

\textit{Function~$g_1$.} Each leaf node fixes both~$\decvar_1$ and~$\decvar_2$ (see Figure~\ref{fig:ex-tree}). Hence, the set~$S$ defined in~\eqref{eq:def-S} consists of the four points~$(y, 1 + y_1 + y_2)$ with~$y \in \{0,1\}^2$. These points lie in the plane~$\tau = 1 + y_1 + y_2$, and so does~$\conv (S)$. Therefore,~$g_1 (\decvar_1, \decvar_2) = 1 + \decvar_1 + \decvar_2$, so that~$g_1 (p) = 2$ and~$g_1 (q) = 7/3$.

\textit{Functions~$g_2$ and~$g_3$.} By Proposition~\ref{prop:value-functions} and Equation~\eqref{eq:inter-bnd1},~$g_2 - 1$ and~$g_3 - 1$ are the minimum values of~$\indvar_{a_1} + \indvar_{b_0} + 2 \indvar_{b_1}$ and of~$\max \{ \indvar_{a_1}, \indvar_{b_0}, 2 \indvar_{b_1} \}$, respectively, over the same flows~$\indvar$, since the nodes outside~$\posnodes$ have the dual bound~$\minlp$. At~$q$, the capacities~$\capa_a (q) = 0$, $\capa_{b_0} (q) = 2/3$, and~$\capa_{b_1} (q) = 1/3$ force~$\indvar_{a_0} = \indvar_{a_1} = 0$, $\indvar_{b_0} = 2/3$, and~$\indvar_{b_1} = 1/3$. Hence,~$g_2 (q) = 1 + 2/3 + 2/3 = 7/3$ and~$g_3 (q) = 1 + 2/3 = 5/3$. At~$p$, all capacities equal~$1/2$. They force~$\indvar_a = \indvar_b = 1/2$, and each subtree splits its flow of~$1/2$ freely between its two leaf nodes. For both functions, the flow~$\indvar_{a_0} = 1/2$, $\indvar_{a_1} = 0$ is optimal in the subtree of~$a$. In the subtree of~$b$, the lower bounds
\begin{equation*}
  \indvar_{b_0} + 2 \indvar_{b_1} \geq \indvar_{b_0} + \indvar_{b_1} = \tfrac12
  \qquad \text{and} \qquad
  \max \{ \indvar_{b_0}, 2 \indvar_{b_1} \} \geq \tfrac23 \indvar_{b_0} + \tfrac13 \cdot 2 \indvar_{b_1} = \tfrac13
\end{equation*}
are attained by~$(\indvar_{b_0}, \indvar_{b_1}) = (1/2, 0)$ and~$(1/3, 1/6)$, respectively. Hence,~$g_2 (p) = 3/2$ and~$g_3 (p) = 4/3$.

\textit{Function~$g_\clos$.} We have~$\permu_1 = b_1$, since~$\lpval_{b_1} = 3 > \lpval_{a_1} = \lpval_{b_0} = 2$. The \STI s~\eqref{eq:new-star-ineq} associated with~$T = \emptyset$, $\{a_1\}$, and~$\{b_0\}$ have the right-hand sides
\begin{equation*}
  3 - 2 \viparam_{b_1} (\decvar), \qquad 3 - \viparam_{b_1} (\decvar) - \viparam_{a_1} (\decvar), \qquad 3 - \viparam_{b_1} (\decvar) - \viparam_{b_0} (\decvar),
\end{equation*}
and~$T = \{a_1, b_0\}$ yields one of the last two, since~$\lpval_{a_1} = \lpval_{b_0}$. The values~$(\viparam_{b_1}, \viparam_{a_1}, \viparam_{b_0})$ are~$(1, 1, 1)$ at~$p$ and~$(2/3, 1, 1/3)$ at~$q$. Hence,~$g_\clos (p) = 1$ and~$g_\clos (q) = \max \{5/3, 4/3, 2\} = 2$.

\printbibliography

\end{document}